\documentclass{amsart}
\allowdisplaybreaks
\usepackage{amsmath,amsthm,amssymb,graphicx,color}
\usepackage{epsfig,amsfonts,latexsym}

\usepackage[colorlinks = false,
linkcolor = black,
urlcolor  = black,
citecolor = black,
anchorcolor = black]{hyperref}

\newcommand{\MYhref}[3][black]{\href{#2}{\color{#1}{#3}}}%

\usepackage{enumerate}
\allowdisplaybreaks
\usepackage{bm}
\newcommand*{\B}[1]{\ifmmode\bm{#1}\else\textbf{#1}\fi}
\usepackage{natbib}
\usepackage[left=1.5in,top=1.5in,right=1.5in, bottom=1.47in]{geometry}

\numberwithin{equation}{section}

\allowdisplaybreaks

\usepackage{pslatex}
\usepackage{graphicx}
\usepackage{hyperref}
\usepackage{lineno}
\usepackage{color}

\def\convas{\stackrel{a.s.}{\longrightarrow}}
\def\eqd{\stackrel{d}{=}}

\def\to{\rightarrow}

\def\<{\langle}
\def\>{\rangle}

\newcommand\mnote[1]{} 
\newcommand{\beq}[1]{\begin{equation}\label{#1}}
\newcommand\eeq{\end{equation}}
\newcommand\ben{\begin{equation}}
\newcommand\een{\end{equation}}
\newcommand\bes{\begin{eqnarray*}}
	\newcommand\ees{\end{eqnarray*}}
\newcommand\besn{\begin{eqnarray}}
\newcommand\eesn{\end{eqnarray}}
\newcommand\beal{\begin{align*}}
\newcommand\eaal{\end{align*}}

\def\bthm{\begin{theorem}}
	\def\ethm{\end{theorem}}
\def\bdefn{\begin{definition}}
	\def\edefn{\end{definition}}
\def\benu{\begin{enumerate}}
	\def\eenu{\end{enumerate}}
\def\beit{\begin{itemize}}
	\def\eeit{\end{itemize}}
\def\beds{\begin{description}}
	\def\eeds{\end{description}}
\def\bepr{\begin{problem}}
	\def\eepr{\end{problem}}

\def\bprf{\begin{proof}}
	\def\eprf{\end{proof}}
\def\berk{\begin{remark}}
	\def\eerk{\end{remark}}
\def\bex{\begin{exercise}}
	\def\eex{\end{exercise}}
\def\beg{\begin{example}}
	\def\eeg{\end{example}}

\def\LLL{{\mathcal L}}

\def\PP{{\mathcal P}}
\def\NN{{\mathcal N}}

\newcommand{\sm}{{\raise0.3ex\hbox{$\scriptstyle \setminus$}}}

\renewcommand{\phi}{\varphi}

\def\CHI{\mathchoice%
	{\raise2pt\hbox{$\chi$}}%
	{\raise2pt\hbox{$\chi$}}%
	{\raise1.3pt\hbox{$\scriptstyle\chi$}}%
	{\raise0.8pt\hbox{$\scriptscriptstyle\chi$}}}
\def\smalloplus{\raise1pt\hbox{$\,\scriptstyle \oplus\;$}}

\usepackage{mathtools}
\usepackage{pslatex}

\def \sas {S$\alpha$S }

\def \1 {\mathbf{1}}

\usepackage{amsmath,amsthm,amssymb,color,tikz,epsfig,amsfonts,latexsym,  bbm, mathrsfs}
\usepackage{epsfig,amsfonts,latexsym}
\usepackage{hyperref}
\hypersetup{
	colorlinks=true,
	linkcolor=blue,
	citecolor=blue,
	urlcolor=blue,
	pdfborder={0 0 0}
}

\usepackage{graphicx}
\usepackage{xcolor}

\numberwithin{equation}{section}

\usetikzlibrary{matrix}

\usepackage{bbm}

\allowdisplaybreaks

\newcommand{\been}{\begin{enumerate}}

	\usepackage{cleveref}
	\newtheorem{thm}{Theorem}[section]
	
	\newtheorem{definition}[thm]{Definition}
	
	\newtheorem{qn}[thm]{Question}
	\newtheorem{op}{Problem}
	\newtheorem{conj}[op]{Conjecture}
	
	\newtheorem{cor}[thm]{Corollary}
	\newtheorem{remark}[thm]{Remark}
	\newtheorem{example}[thm]{Example}

	\usepackage{mathtools}
	\usepackage{pslatex}

	\def \sas {S$\alpha$S }

\begin{document}
	
	\author{Parthanil Roy}
	
	\address{Theoretical Statistics and Mathematics Unit, Indian Statistical Institute, Bangalore, Karnataka 560059, India.}
	
	
	\email{parthanil.roy@gmail.com}

		\title{Group measure space construction, ergodicity and $W^\ast$-rigidity for stable random fields}

		
		\begin{abstract}
			This work discovers a novel link between probability theory (of stable random fields) and von Neumann algebras. It is established that the \emph{group measure space construction} corresponding to a minimal representation is an invariant of a stationary \emph{symmetric $\alpha$-stable} (S$\alpha$S) random field indexed by any countable group $G$. When $G=\mathbb{Z}^d$, we characterize ergodicity (and also absolute non-ergodicity) of stationary S$\alpha$S fields in terms of the central decomposition of this crossed product von Neumann algebra coming from any (not necessarily minimal) Rosi\'{n}ski representation. This shows that ergodicity (or the complete absence of it) is a $W^\ast$-rigid property (in a suitable sense) for this class of fields. All our results have analogues for stationary max-stable random fields as well. 
		\end{abstract}
		
		
		
		\thanks{This research was partially supported by a MATRICS grant from the Science and Engineering Research Board, and a Swarnajayanti Fellowship from Department of Science and Technology, Government of India.}

		\dedicatory{Dedicated to the memory of Professor K.~R.~Parthasarathy}

\maketitle

\tableofcontents

\section{Introduction} \label{sec:intro}

This paper establishes a new connection between probability theory and von Neumann algebras via ergodic theory. We will assume that the random variables discussed here are defined on a common probability space $(\Omega, \mathcal{A}, \mathbb{P})$ unless mentioned otherwise. The corresponding expectation operator will be denoted by $\mathbb{E}(\cdot)$. For two random variables $Y$, $Z$, we write $Y\eqd Z$ if $Y$ and $Z$ are identically distributed.
For any index set $T$ and two stochastic processes (i.e., two collections of random variables defined on $(\Omega, \mathcal{A}, \mathbb{P})$) $\{Y(t)\}_{t \in T}$ and $\{Z(t)\}_{t \in T}$, the notation $\{Y(t)\}  \eqd \{Z(t)\}$ (or simply $Y(t) \eqd Z(t)$, $t \in T$) means that they have the same finite-dimensional distributions. We will take $T=G$ to be a countable group in this paper. 

For $\alpha \in (0, \infty]$ and a $\sigma$-finite standard measure space $(S,\mathcal{S}, \mu)$, we define the space $\LLL^{\alpha}(S,\mu):=\left\{f:S\to\mathbb{C} \mbox{ measurable}~|~
\|f\|_\alpha <\infty \right\}$, where $\|f\|_\alpha:=\left(\int_S|f(s)|^{\alpha}\,\mu(ds)\right)^{1/\alpha}$ for $\alpha \in (0, \infty)$ and $\|f\|_\infty:= \sup_{s \in S} |f(s)|$ (the \emph{supremum} is actually an \emph{essential supremum}). The equalities in $\LLL^{\alpha}(S,\mu)$ (and everywhere else in this paper) should be interpreted as modulo null sets. Note that $\|\cdot\|_\alpha$ is a norm if and only if $\alpha \in [1,\infty]$ making the corresponding $\LLL^{\alpha}(S,\mu)$ a Banach space. For $\alpha \in (0,1)$, however, $\LLL^{\alpha}(S,\mu)$ is an extremely rigid linear space over $\mathbb{C}$ with very few isometries. For our purpose, the functions belonging to $\LLL^{\alpha}(S,\mu)$ for $0 < \alpha <2$ will actually be real-valued. They will arise in the integral representations of a class of real-valued stochastic processes called symmetric $\alpha$-stable random fields, which will be systematically defined in the next two paragraphs.

A random variable $X$ is said to follow \emph{symmetric $\alpha$-stable} (S$\alpha$S) distribution
($\alpha \in (0, 2]$, the index of stability) with scale parameter
$\sigma > 0$ if it has characteristic function (i.e., Fourier transform) of the form
$$
\mathbb{E}(e^{i\theta X}) = \exp{\{-\sigma^{\alpha}|\theta|^\alpha\}}, \;\;\theta \in \mathbb{R}.
$$ 
It can be shown that such random variables exist and arise as scaling limits of sums of symmetric random variables. When $\alpha =2$, we get a Gaussian random variable. In this paper, however, we will always concentrate on the non-Gaussian case, i.e., $\alpha \in (0,2)$. For an encyclopedic treatment of $\alpha$-stable ($0<\alpha <2$) distributions and processes, we refer to \cite{samorodnitsky:taqqu:1994}. 

Let $G$ be a countable group with identity element $e$. A stochastic process $\mathbf{X}=\{X_t\}_{t\in G}$ is called an \sas random field if for each $k \geq 1$, for each $t_1, t_2, \ldots, t_k \in G$ and for each $c_1, c_2, \ldots, c_k \in \mathbb{R}$, the linear combination $\sum_{i=1}^k c_i X_{t_i}$  follows an \sas distribution. Also $\{X_t\}_{t \in G}$ is called left-stationary if $\{X_t\} \overset{d}{=} \{X_{st}\} $ for all $s\in G$. The notion of right-stationarity can be defined similarly and will coincide with left-stationarity when $G$ is abelian. All our results for left-stationary \sas random fields will have their counterparts in the right-stationary case as well. From now on, we shall write stationary to mean left-stationary throughout this work.

In this paper, we introduce a von Neumann algebraic invariant for a stationary \sas random field indexed by any countable group using the crossed product construction of \cite{murray:vonneuman:1936} for the group action arising in the work of \cite{rosinski:1995}. Specializing in the case when the indexing group is $\mathbb{Z}^d$, we give a characterization of ergodicity (and also of complete non-ergodicity) of these fields in terms of the central decomposition of our invariant establishing $W^\ast$-rigidity (in an appropriate sense) of ergodicity (as well as full non-ergodicity) in this setup. We also observe in Section~\ref{sec:max_stable} that all our results hold for stationary max-stable random fields as well.

The rest of this section concentrates on a brief outline of our main contributions and also the machineries used to establish them. In a nutshell, the techniques of the proofs are based on ergodic theory, operator algebra as well as probability theory. To the best of our knowledge, this is the first work on stationary \sas random fields that constructs this crossed product invariant and uses it to establish a  $W^\ast$-rigidity result. We expect this association between von Neumann algebras and stable random fields to be very powerful because both themes are strongly tied up with the ergodic theory of nonsingular group actions as described below. 

It was shown in \cite{rosinski:1995, rosinski:2000} that any stationary \sas random field $\{X_t\}_{t \in G}$ has an integral representation of the form
\begin{equation}
X_t \eqd \int_S c_t(s)\left(\frac{d\mu\circ\phi_t}{d\mu}(s)\right)^{1/\alpha} f\circ \phi_t(s) M(ds), \;\; t \in G, \label{repn:rosinski_intro}
\end{equation}
where $(S, \mathcal{S}, \mu)$ is a $\sigma$-finite standard measure space, $\{\phi_t\}_{t \in G}$ is nonsingular group action on $(S, \mu)$, $\{c_t\}_{t \in G}$ is a $\pm 1$-valued cocycle for $\{\phi_t\}_{t \in G}$, $f \in \LLL^\alpha(S, \mu)$ is a real-valued function and $M$ is an \sas random measure on $S$ with control measure $\mu$. We will assume, without loss of generality, that the full support condition
\[
\bigcup_{t \in G} \mbox{Support}(f \circ \phi_t) = S
\]
holds. See Sections~\ref{sec:erg_th} and \ref{sec:sas} for details of these terminology. Roughly speaking, the randomness in $\{X_t\}_{t \in G}$ is completely absorbed in the random measure $M$, and the infinite-dimensional parameter $\left(f, \{\phi_t\}, \{c_t\}\right)$ (or simply the action $\{\phi_t\}$) carries all (a lot of, resp.) information on the dependence structure of $\{X_t\}_{t \in G}$. 

Keeping the above intuition in mind, it is not at all surprising that various probabilistic facets of a stationary \sas random field $\{X_t\}$ have been connected to ergodic theoretic properties of the underlying nonsingular action $\{\phi_t\}$. These facets include mixing features (see \cite{rosinski:samorodnitsky:1996}, \cite{samorodnitsky:2005a}, \cite{roy:2007a, roy:2012}, \cite{wang:roy:stoev:2013}), large deviations issues (see \cite{mikosch:samorodnitsky:2000a}, \cite{fasen:roy:2016}), growth of maxima (see \cite{samorodnitsky:2004a}, \cite{roy:samorodnitsky:2008}, \cite{owada:samorodnitsky:2015a}, \cite{sarkar:roy:2018}, \cite{athreya:mj:roy:2019}), extremal point processes (see \cite{resnick:samorodnitsky:2004}, \cite{roy:2010a}, \cite{sarkar:roy:2018}), functional central limit theorem (see \cite{owada:samorodnitsky:2015b},  \cite{jung:owada:samorodnitsky:2017}), statistical aspects (see \cite{bhattacharya:roy:2018}), uniform H\"{o}lder continuity of paths (see \cite{panigrahi:roy:xiao:2021}), etc.

On the other hand, given a nonsingular group action $\{\phi_t\}$, it is possible to construct a crossed product von Neumann algebra $\LLL^\infty(S, \mu) \rtimes G$ called the group measure space construction (introduced in the measure-preserving case by \cite{murray:vonneuman:1936}); see Section~\ref{subsec:grp_ms_constr}. It is well-known that ergodic theoretic properties of the nonsingular action $\{\phi_t\}$ are nicely encoded in $\LLL^\infty(S, \mu) \rtimes G$; see, for example, \cite{jones:2009}, \cite{peterson:2013} and the references therein. Therefore, one would expect, in light of the discussions in the previous paragraph, that the group measure space construction corresponding to the underlying group action $\{\phi_t\}$ should become an important invariant that contains a lot of information about probabilistics features of the stable random field $\{X_t\}$. 

The matter is slightly delicate because the integral representation of the form \eqref{repn:rosinski_intro} (now known as the \emph{Rosi\'{n}ski representation}; see \cite{roy:2017}) isn't unique. However, this obstacle can be overcome if we restrict our attention to minimal representations (see Definition~\ref{defn:minml:repn} below). It was shown by \cite{rosinski:1995} for $G=\mathbb{Z}$ that the group actions in any two minimal representations of a fixed stationary \sas random field are conjugate (i.e., isomorphic as group actions). We observe that a careful imitation of the proof, keeping in mind the potential noncommutativity of $G$, extends this result to any countable group. This yields Theorem~\ref{thm:minml_repn:conjugacy}, from which the following result can be shown since conjugacy implies orbit equivalence, which in turn implies $W^\ast$-equivalence; see, for example, \cite{singer:1955}. 

\begin{thm} \label{thm:minml:invariant_intro}
The group actions arising in all minimal representations of a fixed stationary \sas random field are $W^\ast$-equivalent, i.e, their group measure space constructions are isomorphic as von Neumann algebras. 
\end{thm}

\noindent We will call the von Neumann algebra obtained in the above theorem the \emph{minimal group measure space construction} of the stationary \sas random field. This is an important invariant that contains a lot of information on the probabilistic properties of the field. 

Note that \cite{rosinski:1995} established that any minimal representation is actually a Rosi\'{n}ski representation even though the converse is not true. Nor do all Rosi\'nski representations enjoy the uniqueness property as in Theorem~\ref{thm:minml_repn:conjugacy} - two group actions arising out of two different Rosi\'nski representations may not be conjugates of each other. Therefore, the group measure space construction corresponding to a general Rosi\'{n}ski representation may not be an invariant for a stationary \sas random field. In other words, two nonsingular actions arising in two different Rosi\'{n}ski representations (of the same stationary \sas field) may not be $W^\ast$-equivalent. However we do expect, in view of Remark~2.5 of \cite{rosinski:1995} (see also \eqref{reln_between_f_and_fast} below), that some von Neumann algebraic properties will be preserved under further conditions on the group and/or the actions. This is manifested in Corollary~\ref{cor:rosinski_invariant}, for example; see also the discussions below.  

In order to obtain finer results, we assume that $G=\mathbb{Z}^d$ and its action $\{\phi_t\}_{t \in \mathbb{Z}^d}$ (arising in a Rosi\'{n}ski representation of $\{X_t\}_{t \in \mathbb{Z}^d}$) is free (i.e., almost all stabilizers are trivial). The latter assumption enables us to nicely connect the central decomposition (i.e., \eqref{decomp:central} below) of a group measure space construction to the ergodic decomposition wrt the underlying action; see Theorem~\ref{thm:cntrl_n_erg_decomp}, which, together with a characterization of ergodicity of $\{X_t\}_{t \in \mathbb{Z}^d}$ (established in \cite{samorodnitsky:2005a}  for $d=1$ and \cite{wang:roy:stoev:2013} for $d > 1$), gives rise to the following result.

\begin{thm} \label{thm:erg_charac_intro} Suppose $G = \mathbb{Z}^d$ and $\{X_t\}_{t \in \mathbb{Z}^d}$ is a stationary \sas random field. If the nonsingular $\mathbb{Z}^d$-action $\{\phi_t\}$ on $(S, \mu)$ arising in a Rosi\'{n}ski representation of $\{X_t\}$ is free, then the following are equivalent:
\begin{enumerate}
	\item $\{X_t\}$ is ergodic;
	\item $\{X_t\}$ is weakly mixing;
	\item the group measure space construction corresponding to $\{\phi_t\}$ does not admit a $II_1$ factor in its central decomposition.
\end{enumerate}	
\end{thm}

Observe that the first two statements of the above theorem are both ergodic theoretic and probabilistic whereas the third one is operator algebraic in nature. Thus Theorem~\ref{thm:erg_charac_intro} builds a new bridge between probability theory, ergodic theory and operator algebra. The next result strengthens this bridge by giving an analogous von Neumann algebraic criterion for absolute non-ergodicity for the same class of random fields. 

\begin{thm} \label{thm:erg_part_intro} Suppose $G = \mathbb{Z}^d$ and $\{X_t\}_{t \in \mathbb{Z}^d}$ is a stationary \sas random field. If the nonsingular $\mathbb{Z}^d$-action $\{\phi_t\}$ on $(S, \mu)$ arising in a Rosi\'{n}ski representation of $\{X_t\}$ is free, then the following are equivalent:
	\begin{enumerate}
		\item $\{X_t\}_{t \in \mathbb{Z}^d}$ does not have a nontrivial ergodic part in the sense of \cite{wang:roy:stoev:2013};
		\item $\{\phi_t\}_{t \in \mathbb{Z}^d}$ is a positive action, i.e., there exists a $\{\phi_t\}$-invariant probability measure $\tau \sim \mu$ on $S$;
		\item the group measure space construction corresponding to $\{\phi_t\}$ admits only a $II_1$ factor(s) in its central decomposition.
	\end{enumerate}	
\end{thm}
\noindent 

Theorems~\ref{thm:erg_charac_intro} and ~\ref{thm:erg_part_intro} follow from Theorem~4.1 of \cite{wang:roy:stoev:2013} and Theorem~\ref{thm:cntrl_n_erg_decomp} below with the help of various facets of ergodic decomposition as given in Theorem~\ref{thm:erg_comp}. From the two starkly contrasting situations in these two theorems, we can  infer that roughly speaking, presence of more $II_1$ factors in the central decomposition of a {\emph{Rosi\'{n}ski group measure space construction} (i.e., group measure space construction corresponding to a nonsingular action arising in a Rosi\'{n}ski representation) is an indication of weaker ergodicity and hence stronger dependence for stationary \sas random fields indexed by $\mathbb{Z}^d$; see Remark~\ref{remark:LRD} below. In other words, we are carrying forward a program of \cite{samorodnitsky:2004a} (of linking long range dependence for such a field with ergodic theoretic properties of the underlying action) to the realm of operator algebra. 

Additionally, it transpires from Theorems~\ref{thm:erg_charac_intro} and \ref{thm:erg_part_intro} that under our assumptions, even a Rosi\'{n}ski group measure space construction carries full information about ergodicity (resp., about the complete absence of it) for a stationary \sas random field. In other words, if two stationary \sas random fields indexed by $\mathbb{Z}^{d_1}$ and $\mathbb{Z}^{d_2}$ have isomorphic Rosi\'{n}ski group measure space constructions (possibly with $d_1 \neq d_2$ because of \cite{connes:feldman:weiss:1981}), then one of them is ergodic (fully non-ergodic, resp.) if and only if the other one is so. 

Inspired by the discussions in the previous paragraph and the recent progress on $W^\ast$-superrigidity (a term coined by Sorin Popa; see, for instance, the survey paper of \cite{ioana:2018} and the references therein) of an action, we ask the following questions.

\begin{qn} \label{qn_on_superrigidity_intro}
	Does the minimal group measure space construction fully (or partially) remember the stationary \sas random field in some suitable sense? How about a Rosi\'{n}ski group measure space construction?
\end{qn}
\noindent While the complete answer to the above questions would be immensely challenging, they do open a Pandora's box of conjectures and open problems; see Section~\ref{sec:conj}. 

It would be awesome if we can put conditions on one (or both) of the random fields so that it is possible get a positive answer to Question~\ref{qn_on_superrigidity_intro}. We must admit that we are still quite far from doing so, with the only progress being Theorem~\ref{thm:rigidity_intro} below. Motivated by Question~\ref{qn_on_superrigidity_intro} and the rich theory of rigidity (and superrigidity) for actions, we introduce two related notions of $W^\ast$-rigidity (namely, $W^\ast_m$- and $W^\ast_R$-rigidities corresponding to \emph{minimal} and \emph{Rosi\'{n}ski} representations, respectively) for a property of stationary \sas fields (see Section~\ref{sec:Wstar_sas}) and establish the following result. 

\begin{thm} \label{thm:rigidity_intro} Ergodicity (equivalently, weak mixing) is a $W^\ast_R$-rigid (and hence a $W^\ast_m$-rigid) property for stationary \sas random fields indexed by $\mathbb{Z}^d$. Complete lack of ergodicity (as described in Theorem~\ref{thm:erg_part_intro}) is also a $W^\ast_R$-rigid (and hence a $W^\ast_m$-rigid) property for such random fields. In particular, both ergodicity and absolute non-egodicity are orbit equivalence rigid properties as well. 
\end{thm}

\noindent We actually conjecture (see Section~\ref{sec:final_remarks}) that many probabilistic properties of a stable random field are $W^\ast_m$-rigid, making our invariant a powerful one. 

This paper is organized as follows. In Sections~\ref{sec:erg_th} and \ref{sec:vNalg}, we give the prelimineries on ergodic theory (of nonsingular group actions) and operator algebra (of group measure space constructions), respectively. On the other hand, Section~\ref{sec:sas} focuses on a brief overview of stationary \sas random fields. Section~\ref{sec:main_results} is the key section of this paper, containing all the new results and their proofs. We conclude with Section~\ref{sec:final_remarks}, where we discuss how our results can be applied to a few examples, and state some conjectures and open problems. Some of these problems have recently been resolved; see Remark~\ref{remark:solns} below for the details.

\section{Ergodic Theory of Nonsingular Actions} \label{sec:erg_th}

In this section, we start with a brief overview of ergodic theory of nonsingular group actions. As stated above, any statement related to measure spaces should be thought of as true modulo null sets. Since the underlying group $G$ is countable, this will not lead to any measure theoretic difficulty. We start with the definition of nonsingular (also known as quasi-invariant) group action.
 
Recall that $G$ is a countable group with identity element $e$ and $(S, \mathcal{S}, \mu)$ is a $\sigma$-finite standard measure space. A collection of measurable maps $\phi_t : S \rightarrow S$ indexed by $t \in G$ is called a group action of $G$ on $S$ if
\begin{itemize}
	\item[1.] $\phi_e$ is the identity map on $S$, and
	\item[2.] $\phi_{uv} = \phi_v  \circ \phi_u$ for all $u, v \in G$.
\end{itemize}
Note that the order in which the two maps $\phi_v$ and $\phi_u$ appear in the above definition is important because $G$ is possibly noncommutative. In the usual notation for group actions, $\phi_t: s \mapsto (t^{-1}).s$ for each $t \in G$. 
\begin{definition}
A group action $\{\phi_t \}_{t\in G}$ of $G$ on $S$ is called \textbf{nonsingular} (also known as \textbf{quasi-invariant}) if $\mu \circ \phi_t \sim \mu$ for all $t \in G$. Here ``$\sim$'' denotes equivalence of measures.
\end{definition}
\noindent We refer to \cite{varadarajan:1970}, \cite{zimmer:1984}, \cite{krengel:1985} and \cite{aaronson:1997} for the ergodic theory of such actions. Note that measure-preserving group actions $\{\phi_t \}$ (i.e., $\mu \circ \phi_t = \mu$ for all $t \in G$) are clearly nonsingular but the converse is not true.

\subsection{Ergodic and Neveu Decompositions} \label{sec:two_decomp}

Recall that a nonsingular group action $\{\phi_t\}_{t\in G}$ on $(S, \mu)$ is called ergodic if its invariant $\sigma$-field is $\mu$-trivial, i.e., whenever $A \subseteq S$ is such that $A=\phi_t(A)$ (modulo $\mu$) for all $t \in G$, then either $\mu(A)=0$ or $\mu(A^c)=0$. The following result states that even when a nonsingular action is not ergodic, it can be decomposed into ``ergodic components'' in a ``measurable way''. See, for example, Corollary~6.9 of \cite{schmidt:1977}.

\begin{thm} [Existence of Ergodic Decomposition] \label{thm:erg_comp} Let $\{\phi_t \}_{t\in G}$ be a nonsingular action of a countable group $G$ on a $\sigma$-finite standard measure space $(S, \mathcal{S}, \mu)$. Then there exists another $\sigma$-finite standard Borel space $(Y, \mathcal{Y})$, a measurable map $\Psi:S \to Y$,  and a family $\{\mu_y: y \in Y\}$ of $\sigma$-finite measures on $(S, \mathcal{S})$ such that
\begin{enumerate}
\item for each $B \in \mathcal{S}$, $y \mapsto \mu_y(B)$ is a measurable map of $Y$ into $[0, \infty]$;
\item for each $B \in \mathcal{S}$, $$\mu(B) = \int_Y \mu_y(B) d\nu(y),$$
          where $\nu := \mu \circ \Psi^{-1}$;
\item $\mu_y\big(\Psi^{-1}(\{y\}^c)\big)=0$ for each $y \in Y$ (in particular, $\mu_y$ and $\mu_{y^\prime}$ are mutually singular whenever $y \neq y^\prime$);
\item $\{\phi_t \}_{t\in G}$ is a nonsingular and ergodic action restricted to each $(S_y, \mu_y)$, where $S_y:=\Psi^{-1}(\{y\})$;
\item if $\mu$ is $\{\phi_t \}_{t\in G}$-invariant, then so is each $\mu_y$;
\item if $\mu$ is equivalent to a $\{\phi_t \}_{t\in G}$-invariant $\sigma$-finite measure $\tau$, then for each $y \in Y$, $\mu_y \sim \tau|_{S_y}$.
\end{enumerate}
\end{thm}

\begin{definition} \label{defn:erg_decomp}
Suppose $\{\phi_t \}_{t\in G}$ is a nonsingular action of a countable group $G$ on a $\sigma$-finite standard measure space $(S, \mathcal{S}, \mu)$. Then the collection $\{(S_y, \mu_y): y \in Y\}$ obtained in Theorem~\ref{thm:erg_comp} is called the ergodic decomposition of $S$ wrt the action $\{\phi_t \}$.
\end{definition}	
		
It can be shown that the ergodic decomposition is unique up to isomophism. In some sense, it gives the ``finest possible'' partition of $S$ into invariant components. Another partition (possibly much coarser) into two invariant components is given by the Neveu decomposition, which is described below. This decomposition is closely connected to ergodicity (as well as to complete non-ergodicity) of the stationary S$\alpha$S random fields indexed by $G=\mathbb{Z}^d$. Hence, it is the backbone of proving a few results of this paper. 

The Neveu decomposition is obtained by applying Lemma~2.2 and Theorem 2.3~(i) in \cite{wang:roy:stoev:2013} in the case of any countable group $G$ (not just $\mathbb{Z}^d$). This yields the partition  $S=\PP\cup \NN$, where the set $\PP$ is the largest (modulo $\mu$) $\{\phi_t\}$-invariant set where one can have a finite measure preserved by $\{\phi_t\}$ and equivalent to $\mu$, and $\NN=\PP^c$.  The subsets $\PP$ and $\NN$ of $S$ are known as the \emph{positive part} and the \emph{null part} of $\{\phi_t\}_{t \in G}$, respectively.

\subsection{Conjugacy and Orbit Equivalence} \label{sec:two_equiv_reln}

In this subsection, we define two equivalence relations on the space of all nonsingular group actions. One more will be introduced in Section~\ref{sec:Wstar_super}. All of these will play significant roles in this work. 

\begin{definition} \label{defn:conj_oe}
	Let $\{\phi^{(i)}_t\}_{t \in G_i}$ be a nonsingular action of a countable group $G_i$ on a $\sigma$-finite standard Borel space $(S_i, \mathcal{S}_i, \mu_i)$ for each $i=1, \, 2$. These two actions  $\{\phi^{(1)}_t\}$ and $\{\phi^{(2)}_t\}$ are called
	
	\begin{enumerate}
		\item \textbf{conjugate} if there exist a group isomosphism $\alpha: G_1 \to G_2$ and a bijection $h: S_1 \to S_2$ such that both $h$ and $h^{-1}$ are measurable, $\mu_1 \circ h^{-1} \sim \mu_2$, and for all $t \in G$,
		\begin{equation}
		h \circ \phi^{(1)}_t = \phi^{(2)}_{\alpha(t)} \circ h \; \;\; ; \label{eqn_conjugacy}
		\end{equation}
		
		\item \textbf{orbit equivalent} if there exist a group isomosphism $\alpha: G_1 \to G_2$ and a bijection $h: S_1 \to S_2$ such that both $h$ and $h^{-1}$ are measurable, $\mu_1 \circ h^{-1} \sim \mu_2$, and for $\mu_1$-almost all $s_1 \in S_1$,
		\begin{equation}
		h\big(\{\phi^{(1)}_t(s_1): t \in G\}\big) =\big\{\phi^{(2)}_{\alpha(t)}(h(s_1)): t \in G\big\}. \label{eqn_orbit_equiv}
		\end{equation}
	\end{enumerate}
Furthermore, if $G_1 = G_2$, then unless mentioned otherwise, $\alpha$ is taken as the identity isomorphism in the above two definitions. 
\end{definition}

Clearly \eqref{eqn_orbit_equiv} follows from \eqref{eqn_conjugacy} and hence conjugacy implies orbit equivalence, but the converse does not hold. One should think of conjugacy as the isomorphism in the category of nonsingular group actions while the orbit equivalence is a much weaker notion that only demands correspondence of orbits. As we will see, an even weaker notion of $W^\ast$-equivalence (see Definition~\ref{defn:Wstar_equiv} below) is useful enough in ergodic theory and operator algebra, and hence for stable fields as well. The equivalence relations in Definition~\ref{defn:conj_oe} can also be defined for measure-preseving actions in a similar fashion although in this case, the measures $\mu_1 \circ h^{-1}$ and $\mu_2$ need to be equal, not just equivalent.

\section{von Neumann Algebras} \label{sec:vNalg}
	
This section is devoted to von Neumann algebras and more specifically, to the  \emph{group measure space construction} corresponding to a nonsingular action. For detailed discussions on von Neumann algebras and proofs of the results stated in this section, we refer the readers to \cite{bratteli:robinson:1987}, \cite{sunder:1987}, \cite{jones:2009}, \cite{peterson:2013} and the references therein.

Let $\mathcal{H}$ be a separable Hilbert space over $\mathbb{C}$ and $\mathcal{B}(\mathcal{H})$ be the space of all bounded operators on $\mathcal{H}$. The following elegant result initiated the study of von Neumann algebras.

\begin{thm}[von Neumann's Bicommutant Theorem] \label{thm:bicomm} Suppose $M$ is a $\ast$-subalgebra of $\mathcal{B}(\mathcal{H})$ containing $1$, the identity operator. Then the following are equivalent:
\begin{enumerate}
  \item $M$ is closed in the weak operator topology.
  \item $M$ is closed in the strong operator topology.
  \item $M=(M^\prime)^\prime =: M^{\prime\prime}$.
\end{enumerate}
Here $M^\prime := \{T \in \mathcal{B}(\mathcal{H}): TA = AT \mbox{ for all } A \in M\}$ is the commutant of $M$.
\end{thm}

\noindent Observe that the first two are analytic properties while the third one is an algebraic one. Thus, von Neumann's beautiful result binds the two subjects nicely and gives rise to the following important notion, which can also be defined more abstractly (not just as a subalgebra of $\mathcal{B}(\mathcal{H})$) but the more concrete definition below will serve our purpose.

\begin{definition} A unital $\ast$-subalgebra of $\mathcal{B}(\mathcal{H})$ satisfying one (and hence all) of the equivalent conditions in  Theorem~\ref{thm:bicomm} is called a von Neumann algebra.
\end{definition}

\subsection{Factors and Central Decomposition}

Note that if $M$ is a von Neumann algebra, then so is $M^\prime$. We now define a very important special case that serves as a building block in the investigation of von Neumann algebras.

\begin{definition} A von Neumann algebra $M$ is called a factor if $Z(M):=M \cap M^\prime = \mathbb{C} 1$ (i.e., the centre is trivial).
\end{definition}

\noindent The following result of von Neumann states that any von Neumann algebra can be decomposed as a direct sum (or more generally, direct integral) of factors in a unique fashion.

\begin{thm}[von Neumann] \label{thm:central_decomp} Let $M$ be a von Neumann algebra. Then there exists a $\sigma$-finite measure space $(Y, \mathcal{Y}, \nu)$ and von Neumann algebras $\{M_y: y \in Y\}$ such that the direct integral decomposition
\begin{equation}
  M = \int_Y M_y\, \nu(dy)  \label{decomp:central}
\end{equation}
holds and for $\nu$-almost all $y \in Y$, $M_y$ is a factor. Moreover, this decomposition is almost surely unique up to isomorphism.
\end{thm}
See, for example, \cite{knudby:2011} for a nice exposition of direct integrals of von Neumann algebras and a proof of Theorem~\ref{thm:central_decomp}. This decomposition is known as the central decomposition and it reduces the study of von Neumann algebras to the study of factors. As will be seen later, the central decomposition will play a key role in this paper. We would also like to mention that if $Y$ is countable with $\nu$ being the counting measure, then the direct integral in \eqref{decomp:central} reduces to a direct sum.

It is well-known that any von Neumann algebra $M$ has a lot of projections, which generate $M$. It is possible to define an equivalence relation on these projections such that the equivalence classes form a partially ordered set, which becomes totally ordered whenever $M$ is a factor. Depending on the order structure of this totally ordered set and existence/non-existence of trace (defined below), factors can be classified into various types; see, for example, \cite{sunder:1987}, \cite{jones:2009}, \cite{peterson:2013} for details. In this paper, we shall only need the following type of factors with infinitely many (in fact, uncountably many) equivalence classes of projections. 

\begin{definition} A factor $M$ is said to be of type $II_1$ if it is infinite-dimensional and it admits a normalized trace, i.e., there exists an ultraweakly continuous linear functional $tr: M \to \mathbb{C}$ satisfying $tr(1)=1$, $tr(ab) = tr(ba)$ and $tr(a^\ast a) \geq 0$ for all $a, b \in M$.
\end{definition}

At this point, we would like to mention that commutative von Neumann algebras are just the $\LLL^\infty$-spaces and hence the study of von Neumann algebras is regarded as ``non-commutative measure theory''. On the other hand, the investigation of $II_1$ factors form an integral part of ``non-commutative probability theory''. Therefore, it is not at all unreasonable to expect that these rich structures will be closely tied up with the probabilistic properties of stable random fields. In particular, the following terminology coined in this paper will be essential in characterizing ergodicity (as well as the absolute lack of it) for stationary S$\alpha$S random fields.

\begin{definition} \label{no_II_1_factor} (1)~We say that a von Neumann algebra $M$ does not admit a $II_1$ factor in its central decomposition \eqref{decomp:central} if for $\nu$-almost all $y \in Y$, $M_y$ is a not a $II_1$ factor. 
	
\noindent (2)~On the other hand, we say that $M$ admits only $II_1$ factor(s) in its central decomposition \eqref{decomp:central} if for $\nu$-almost all $y \in Y$, $M_y$ is a $II_1$ factor.
\end{definition}

\noindent As we shall see in Theorem~\ref{thm:erg_charac}, a stationary S$\alpha$S random field indexed by $G=\mathbb{Z}^d$ will be ergodic (completely non-ergodic, resp.) if and only if a specific von Neumann algebra \emph{does not admit a $II_1$ factor} (resp., \emph{admits only $II_1$ factor(s)}) in its central decomposition. This von Neumann algebra is defined in the next subsection along with a discussion on its relation to ergodic theoretic properties of nonsingular group actions.

\subsection{Group Measure Space Construction} \label{subsec:grp_ms_constr}

Given a nonsingular group action $\{\phi_t\}_{t \in G}$ on a $\sigma$-finite standard measure space $(S, \mathcal{S}, \mu)$, we can also construct a von Neumann algebra that reflects the ergodic theoretic properties of the action. This was first introduced by \cite{murray:vonneuman:1936} in the context of measure-preserving group actions.

The $G$-action $\{\phi_t\}$ lifts to the space of all real-valued measurable functions on $S$ by
\[
\sigma_t g = g \circ \phi_t, \;\, t \in G.
\]
This lifted action preserves the $\LLL^\infty$-norm but not other $\LLL^p$-norms. However, for each $t \in G$, $\pi_t: \LLL^2(S, \mu) \to \LLL^2(S, \mu)$ given by
\[
(\pi_t g)(s) = g \circ \phi_t(s) \left(\frac{d \mu \circ \phi_t}{d\mu}(s)\right)^{1/2}, \; s \in S
\]
defines an isometry. The unitary representation $\{\pi_t\}_{t \in G}$ of $G$ inside $\LLL^2(S, \mu)$ is called the Koopman representation.

Using the cocycle relationship
\[
\frac{d \mu \circ \phi_{uv}}{d\mu} = \frac{d \mu \circ \phi_{u}}{d\mu} \, \sigma_u\left(\frac{d \mu \circ \phi_{v}}{d\mu}\right), \;\; u, v \in G,
\]
one gets that for all $a \in \LLL^\infty(S, \mu)$ (thought of as acting on $\LLL^2(S, \mu)$ by multiplication), for all $t \in G$ and for all $g \in \LLL^2(S, \mu)$,
\begin{equation}
(\pi_t \, a \, \pi_{t^{-1}} g)(s) = ((\sigma_t a) g)(s), \;\; s \in S. \label{eqn:crossed_product}
\end{equation}
In other words, the Koopman representation ``normalizes'' $\LLL^\infty(S, \mu)$ inside $\mathcal{B}(\LLL^2(S, \mu))$. The group measure space construction is a space where the crossed product relation \eqref{eqn:crossed_product} is internalized (see \eqref{eqn:crossed_product_int} below). The details of this construction are given here.

Consider the von Neumann algebra 
$$
\mathcal{B}(l^2(G) \otimes \LLL^2(S, \mu)) = \overline{\mathcal{B}(l^2(G)) \otimes \mathcal{B}(\LLL^2(S, \mu))}
$$ 
(with the closure being taken with respect to the weak/strong operator topology). Define a representation of $G$ by $t \mapsto u_t:= \lambda_t \otimes \pi_t$, where $\{\lambda_t\}$ is the left regular representation and $\{\pi_t\}$ is the Koopman representation. We also represent $\LLL^\infty(S, \mu)$ by $a \mapsto 1 \otimes \mathcal{M}_a$, where $\mathcal{M}_a$ is the multiplication (by $a$) operator on $\LLL^2(S, \mu)$. It can be checked that the following ``internal'' crossed product relation holds:
\begin{equation}
u_t (1 \otimes \mathcal{M}_a) u_{t^{-1}} = 1 \otimes \mathcal{M}_{\sigma_t a}\,.  \label{eqn:crossed_product_int}
\end{equation}
Define the \emph{group measure space construction} (also known as \emph{crossed product construction}) as
\[
\LLL^\infty(S, \mu) \rtimes G := \{u_t, 1 \otimes \mathcal{M}_a: t \in G, \, a \in \LLL^\infty(S, \mu)\}^{\prime\prime}.
\]

It can be shown that the crossed product relation \eqref{eqn:crossed_product_int} implies that any $x \in \LLL^\infty(S, \mu) \rtimes G$ can be uniquely written as $x = \sum_{t \in G} a_t u_t$ with $\{a_t: t \in G\} \subseteq \LLL^\infty(S, \mu)$. Thus, we can view $x$ as a $|G| \times |G|$ matrix with entries coming from $\LLL^\infty(S, \mu)$ that are the same along each left group-diagonal. To understand the connection with ergodic theory, let us recall that a nonsingular $G$-action $\{\phi_t\}$ on $(S, \mathcal{S}, \mu)$ is called free if $\mu$-almost all stabilizers are trivial, i.e., 
$$
\mu\big(\{s \in S: \phi_t(s)=s \mbox{ for some }t \neq e\}\big)=0. 
$$
The following result illustrates the relation between ergodic theory and group measure space construction through the underlying nonsingular action.

\begin{thm} \label{thm:factor_ergodic} The following results hold for a nonsingular $G$-action $\{\phi_t\}$ and the corresponding group measure space construction defined above. 
	\begin{enumerate} 
              \item If the action $\{\phi_t\}_{t \in G}$ is free and ergodic, then $\LLL^\infty(S, \mu) \rtimes G$ is a factor.
              \item If $\LLL^\infty(S, \mu) \rtimes G$ is a factor, then $\{\phi_t\}_{t \in G}$ is ergodic.
              \item If $\{\phi_t\}_{t \in G}$ is free and ergodic, then the factor  $\LLL^\infty(S, \mu) \rtimes G$ is of type $II_1$ if and only if $\{\phi_t\}_{t \in G}$ is a positive action (i.e., its null part $\NN$ is of zero $\mu$-measure).
            \end{enumerate}
Furthermore, if the two nonsingular actions (not necessarily of the same group) are orbit-equivalent, then the corresponding group measure space constructions are isomorphic as von Neumann algebras. 
\end{thm}

\noindent The above result shows the strong ties between ergodic theory and von Neumann algebras. In this work, we plan to encash a stronger connection given by the following result, which will play a crucial role in establishing the main result of this paper.

\begin{thm} \label{thm:cntrl_n_erg_decomp} Let $\{\phi_t\}_{t \in G}$ be a free nonsingular action on a $\sigma$-finite standard measure space $(S, \mathcal{S}, \mu)$, which has an ergodic decomposition as given in Theorem~\ref{thm:erg_comp}. Then $M:= \LLL^\infty(S, \mu) \rtimes G$ has a central decomposition \eqref{decomp:central} with $M_y:= \LLL^\infty(S_y, \mu_y) \rtimes G$ for each $y \in Y$. 
\end{thm}

We have already defined conjugacy and orbit equivalence for nonsingular group actions in Section~\ref{sec:two_equiv_reln}. In the next subsection, we will define another equivalence relation called $W^\ast$-equivalence, which has led to the famous notion and the rich theory of $W^\ast$-rigidity as described below.

\subsection{$W^\ast$-rigidity}  \label{sec:Wstar_super}

We start by defining the following equivalence relation (on the space of all nonsingular group actions), which, in spite of being the weakest of the three, will be extremely useful in this work. 

\begin{definition}\label{defn:Wstar_equiv} Two nonsingular group actions (not necessarily of the same group) are called $W^\ast$-equivalent if the corresponding group measure space constructions are isomorphic as von Neumann algebras. 
\end{definition}
\noindent As we have already observed, conjugacy trivially implies orbit equivalence. Theorem~\ref{thm:factor_ergodic} yields that orbit equivalence implies $W^\ast$-equivalence, which was first observed by \cite{singer:1955} for \emph{probability measure-preserving} (p.m.p.) group actions.

Summarizing, we get 
\begin{equation}
\mbox{conjugacy} \;\; \Longrightarrow \;\; \mbox{orbit equivalence} \;\; \Longrightarrow \;\; W^\ast\mbox{-equivalence}.  \label{impli_conj_oe_wstar}
\end{equation}
In general, neither of the above implications can be reversed. A rigidity phenomenon corresponds to general conditions on the groups and/or actions which ensures that the reverse implication(s) hold. More precisely, if under some conditions, $W^\ast$-equivalence implies conjugacy (or at least orbit equivalence), then it will be an example of a rigidity property. Informally speaking, rigidity means that the group measure space construction ``remembers'' the group and its action very well. The strongest form of this is the notion of $W^\ast$-superrigidity (a term coined by Sorin Popa; see, e.g., \cite{ioana:2018} and the references therein) in the context of p.m.p.\ group actions. 

\begin{definition} A p.m.p.\ group action $\{\phi_t\}_{t \in G}$ is called $W^\ast$-superrigid if any free ergodic p.m.p.\ action $W^\ast$-equivalent to $\{\phi_t\}_{t \in G}$ must be conjugate to it.
\end{definition}

Essentially, $W^\ast$-superrigidity of an action means that the conjugacy class of the action can be fully reconstructed from the isomorphism class of its crossed product von Neumann algebra. Roughly speaking, the group measure space construction ``fully remembers'' the group as well as its action. Therefore, this property is incredibly useful in the classification theory and has become a prominent topic of recent research; see \cite{popa:2006}, \cite{peterson:2010}, \cite{popa:vaes:2010, popa:vaes:2014}, \cite{ioana:2011}, etc. Just to give an example, it was shown by \cite{ioana:2011} that the Bernoulli action of any ICC group having Kazhdan's property (T) is $W^\ast$-superrigid. There has been progress towards $W^\ast$-superrigidity for nonsingular (but not necessarily p.m.p.) actions as well; see, for instance, Proposition~D of \cite{vaes:2014}. For a detailed survey of these results, we refer the readers to \cite{ioana:2018}. 

As mentioned in Section~\ref{sec:intro}, to each stationary \sas random field indexed by a countable group, we attach a novel invariant; we term this the \textit{minimal group measure space construction}. Of course, motivated by the notion of $W^\ast$-superrigidity, the natural question to ask would be the following:
\begin{qn} \label{qn_on_superrigidity}
Does the minimal group measure space construction fully remember the stationary \sas random field in some suitable sense? 
\end{qn}
\noindent The complete answer to the above question is still unknown and perhaps extremely difficult. Finding conditions on one of the random fields so that we can get a positive answer to Question~\ref{qn_on_superrigidity} should be thought of as a $W^\ast$-superrigidity question in this setup. 

While $W^\ast$-superrigidity signifies complete memory of the group action, it is quite possible in many situations that only certain features of the action will be remembered. This phenomenon is typically referred to as $W^\ast$-rigidity. We will establish such a result for stationary \sas random fields. More specifically, we will show, with the help of Theorem~\ref{thm:cntrl_n_erg_decomp}, that ergodicity (as well as complete non-ergodicity) will be remembered provided the indexing group of the field is $\mathbb{Z}^d$ and its actions are free. We will actually introduce two types of $W^\ast$-rigidity and show that ergodicity and absolute non-ergodicity are indeed rigid in both senses under our conditions. One must contrast this with the complete lack of rigidity for free ergodic p.m.p.\ actions of $\mathbb{Z}^d$ (more generally, countably infinite amenable groups; see \cite{connes:1976}). Of course, we do allow actions that are not p.m.p.\ or else by Theorem~\ref{thm:erg_charac}, no $\mathbb{Z}^d$-indexed stationary \sas random field would be ergodic.

\section{Stationary S$\alpha$S Random Fields} \label{sec:sas}

In this section, we present the background on symmetric $\alpha$-stable random fields (indexed by a countable group $G$) and their integral representations. We will also introduce the notion of the minimal representation, which will be very useful in the next section. Finally, we specialize to the stationary case and discuss Rosi\'{n}ski representations. For a survey of these results in the $G=\mathbb{Z}$ case, see, for instance, \cite{roy:2017}. 

\subsection{Integral Representation}

Let $\mathbf{X}=\{X_t\}_{t \in G}$ be an \sas ($0< \alpha <2$) (not necessarily stationary) random field indexed by $G$ as defined in Section~\ref{sec:intro}. Any such random field has an \emph{integral representation} (also called \emph{spectral representation}) of the type
   \begin{equation}\label{integrep}
   X_t \overset{d}{=} \int_S f_t(s)M(ds), \mbox{ \ \  } t \in G,
   \end{equation}
where $M$ is an \sas random measure on some standard Borel space $S$ with a $\sigma$-finite control measure $\mu$, and $\{f_t: t \in G\} \subset \mathcal{L}^\alpha (S, \mu)$ is a family of real-valued functions. See, for instance, Theorem 13.1.2 of \cite{samorodnitsky:taqqu:1994}. This simply means that each linear combination $\sum_{i=1}^k c_i X_{t_i}$  follows an \sas distribution with scale parameter $\|\sum_{i=1}^k c_i f_{t_i}\|_\alpha$. The collection $\{f_t\}_{t \in G} \subset \mathcal{L}^\alpha (S, \mu)$ is called a spectral representation or an integral representation of $\{X_t\}_{t \in G}$. We shall assume, without loss of generality, that the full support condition 
\[
\bigcup_{t \in G} \mbox{Support}(f_t) = S
\]
holds for all integral representations $\{f_t\}_{t \in G}$ of $\{X_t\}_{t \in G}$. Note that given any $\sigma$-finite standard measure space $(S, \mathcal{S}, \mu)$, a family of real-valued functions $\{f_t\}_{t \in G} \subset \mathcal{L}^\alpha(S, \mu)$ and an S$\alpha$S random measure $M$ on $S$ with control measure $\mu$, one can construct an S$\alpha$S random field using \eqref{integrep}.

\subsection{Minimal Representation}

The next notion (introduced in \cite{hardin:1982b}) is slightly technical albeit very useful. We shall first give the formal definition and then have a discussion that will help us understand its meaning.

\begin{definition} \label{defn:minml:repn} An integral representation $\{f_t\}_{t \in G} \subset \mathcal{L}^\alpha(S, \mu)$ of an S$\alpha$S random field is called a minimal representation if for all $B \in \mathcal{S}$, there exists $A \in \sigma\big\{f_t/f_{u}: t, u \in G\big\}$ such that $\mu(A \Delta B)=0$.
\end{definition}

\noindent The ratio $f_t(s)/f_{u}(s)$ is defined to be $\infty$ when $f_t(s) \geq 0$, $f_{u}(s)=0$ and $-\infty$ when $f_t(s)<0$, $f_{u}(s)=0$. In particular, the $\sigma$-algebra $\sigma\big\{f_t/f_{u}: t, u \in G\big\}$ is generated by a bunch of extended real-valued functions. It was shown by \cite{hardin:1981, hardin:1982b} that every S$\alpha$S random field has a minimal representation even though it is never unique.

The following discussion provides better insight into the notion of minimality of integral representations. Let $\{f^\ast_t\}_{t \in G} \subset \LLL^\alpha(S^\ast, \mu^\ast)$ be a minimal representation of an S$\alpha$S random field $\{X_t\}_{t \in G}$ and $\{f_t\}_{t \in G} \subset \LLL^\alpha(S, \mu)$ be any integral representation of  $\{X_t\}_{t \in G}$. Then there exist measurable functions $\Phi: S \to S^\ast$ and $\eta: S \to \mathbb{R}\setminus \{0\}$ such that
\begin{equation}
\mu^\ast(A)=\int_{\Phi^{-1}(A)}|\eta|^\alpha d\mu, \;\,\;\;\; A \subseteq {S}^\ast, \label{reln_between_mu_h_muast}
\end{equation}
and for each $t \in G$,
\begin{equation}
f_t(s)=\eta(s) f_t^\ast(\Phi(s)) \;\mbox{ for $\mu$-almost all } s \in S. \label{reln_between_f_and_fast}
\end{equation}

In other words, minimal representations are minimal in the sense that any integral representation can be expressed in terms of them. If further $\{f_t\}_{t \in G}$ above is also a minimal representation, then $\Phi$ and $\eta$ are unique modulo $\mu$, $\Phi$ is one-to-one and onto, $\mu^\ast \circ \Phi \sim \mu$ and
\begin{equation}
|\eta|=\left(\frac{d (\mu^\ast \circ \Phi)}{d\mu}\right)^{1/\alpha}\;\;\mbox{ $\mu$-almost surely.}\label{reln_between_mu_h_muast_f_minimal}
\end{equation}
The proofs of these observations use \emph{analytic rigidity} (i.e., dearth of isometries) of $\LLL^\alpha$-spaces for $\alpha \in (0,2)$; see \cite{hardin:1981, hardin:1982b}. See also \cite{Rosinski:1994, rosinski:1995}.

\subsection{The Stationary Case: Rosi\'nski Representation}

Now assume that $\{X_t\}_{t \in G}$ is (left) stationary as described in Section~\ref{sec:intro}. In this case, we would need one more notion, namely, that of a cocycle defined below.

\begin{definition}
	Suppose $\{\phi_t\}_{t \in G}$ is a nonsingular action on a standard measure space $(S, \mathcal{S}, \mu)$. Then a collection of measurable maps $\big\{c_t: S \to \{+1, -1\}\big\}_{t \in G}$ is called a $\pm 1$-valued cocycle for $\{\phi_t\}$ if 
     for all $t_1, t_2 \in G$, 
     $$c_{t_1 t_2}(s) = c_{t_1} (s)c_{t_2}(\phi_{t_1}(s))$$ for $\mu$-almost all $s \in S$.
\end{definition}
\noindent Using \eqref{reln_between_f_and_fast}, one can show (see \cite{Rosinski:1994}, \cite{rosinski:1995}, \cite{rosinski:2000}, \cite{sarkar:roy:2018}) that any minimal representation of $\{X_t\}_{t \in G}$ has the following special form:
\begin{equation}\label{eq1}
 f_t(s)=c_t(s)\left(\frac{d\mu\circ\phi_t}{d\mu}(s)\right)^{1/\alpha} f\circ \phi_t(s), \mbox{\ \ } t \in G,
\end{equation}
where $f\in \mathcal{L}^{\alpha}(S,\mu)$ is a real-valued function, $\{\phi_t\}_{t \in G}$ is a nonsingular $G$-action on $(S, \mu)$, and $\{c_t\}_{t\in G}$ is a $\pm 1$-valued cocycle for $\{\phi_t\}_{t \in G}$. The relation $\phi_{uv} = \phi_v \circ \phi_u$ comes from the left-stationarity of $\{X_t\}$. It will be reversed if we consider a right-stationary \sas field.

Conversely, given any $\sigma$-finite standard measure space $(S, \mathcal{S}, \mu)$, a real-valued function $f \in\mathcal{L}^\alpha(S, \mu)$, a nonsingular $G$-action $\{\phi_t\}_{t \in G}$ on $(S, \mu)$, a $\pm 1$-valued cocycle $\{c_t\}$ for $\{\phi_t\}$, and an S$\alpha$S random measure $M$ on $S$ with control measure $\mu$, one can construct a stationary S$\alpha$S random field using \eqref{integrep} and \eqref{eq1}. In this case, we say that the stationary \sas random field $\{X_t\}_{t\in G}$ is generated by the nonsingular $G$-action $\{\phi_t\}$. Following \cite{roy:2017}, we shall call any integral representation (not necessarily minimal) of the form \eqref{eq1} a \emph{Rosi{\'n}ski representation} of $\{X_t\}_{t \in G}$. We introduce the following terminology for the ease of presentation of Theorem~\ref{thm:minml_repn:conjugacy} below. 

\begin{definition}\label{defn:triplets}
We will call $\left(f, \{\phi_t\}, \{c_t\}\right)$ a Rosi{\'n}ski triplet (corresponding to the representation \eqref{eq1}) of $\{X_t\}$ on $(S, \mu)$. If further  \eqref{eq1} is a minimal representation, then $\left(f, \{\phi_t\}, \{c_t\}\right)$ will be called a minimal triplet (of $\{X_t\}$ on $(S, \mu)$), or we will simply say that the Rosi{\'n}ski triplet $\left(f, \{\phi_t\}, \{c_t\}\right)$ is minimal. 
\end{definition}

Observe that using our language, \cite{rosinski:1995} actually established that any minimal representation is a Rosi{\'n}ski representation. However, the converse may not hold; the integral representation considered in Remark~4.3 of \cite{roy:2010a} will work as a counter-example (this was provided to the author by Jan  Rosi{\'n}ski).

\section{Linking Stable Fields with von Neumann Algebras} \label{sec:main_results}

As will be seen in this section, the group measure space construction corresponding to all minimal representations (of a fixed stationary S$\alpha$S field) will be isomorphic as von Neumann algebras making this an invariant of the random field. This result opens up a close connection between probabilistic properties of stationary S$\alpha$S random fields and von Neumann algebraic aspects of the group measure space construction corresponding to some (equivalently, any) minimal representation. 

\subsection{Minimal Representations, Conjugacy and the Crossed Product Invariant}

Our first result is an extension of two theorems in \cite{rosinski:1995} put together. It states, among other things, that two nonsingular actions arising in two minimal representations (of a fixed stationary \sas random field) are conjugate. In order to present the precise statement of this result in its full strength, we need to introduce one more equivalnce relation as follows.

\begin{definition} \label{defn:equiv_of_pairs}
	Suppose $\{\phi^{(i)}_t\}_{t \in G}$ is a nonsingular action of a countable group $G$ on a $\sigma$-finite standard Borel space $(S_i, \mathcal{S}_i, \mu_i)$ and $\{c^{(i)}_t\}_{t \in G}$ is a $\pm 1$-valued cocycle for $\{\phi^{(i)}_t\}$ for each $i=1, \, 2$. We write $\left(\{\phi^{(1)}_t\}, \{c^{(1)}_t\}\right) \approx \left(\{\phi^{(2)}_t\}, \{c^{(2)}_t\}\right)$ (and say that the pairs are equivalent) if
	\begin{enumerate}
	\item the $G$-actions $\{\phi^{(1)}_t\}$ and $\{\phi^{(2)}_t\}$ are conjugate via conjugacy map $h:S_1 \to S_2$ as in Definition~\ref{defn:conj_oe} (in particular, \eqref{eqn_conjugacy} holds with $\alpha$ being the identity map on $G$), and
	\item the cocycle $\{c^{(1)}_t \circ h^{-1}\}_{t \in G}$ is cohomologous to $\{c^{(2)}_t\}_{t \in G}$\, , i.e., there exists a measurable map $b:S_2 \to \{+1, -1\}$ such that for each $t \in G$,
	\[
	c^{(1)}_t \circ h^{-1}(s) = c^{(2)}_t(s) \, \frac{b \circ \phi^{(2)}_t(s)}{b(s)}
	\]
	for $\mu_2$-almost all $s \in S_2$. 
    \end{enumerate}
\end{definition}

It is easy to verify that ``$\approx$'' defines an equivalence relation (for pairs of the form (action, cocycle) as above), which holds trivially when the actions are conjugate via the map $h$ and $c^{(1)}_t \circ h^{-1}=c^{(2)}_t$ for all $t \in G$. The following result is a combined generalization of Proposition~3.3 and Theorem~3.6 of \cite{rosinski:1995}, who considered the case $G=\mathbb{Z}$. 

\begin{thm} \label{thm:minml_repn:conjugacy}
	Let $G$ be a countable group and $\{X_t\}_{t \in G}$ be a stationary \sas random field. Then the following results hold.\\
	\noindent (a) Suppose  $\left(f^{(1)}, \{\phi^{(1)}_t\}, \{c^{(1)}_t\}\right)$ is a Rosi\'{n}ski triplet of $\{X_t\}$ on $(S_1, \mu_1)$ as in Definition~\ref{defn:triplets}. Borrowing the notation from  Definition~\ref{defn:equiv_of_pairs} above, we assume that 
	\begin{equation}
	\left(\{\phi^{(1)}_t\}, \{c^{(1)}_t\}\right) \approx \left(\{\phi^{(2)}_t\}, \{c^{(2)}_t\}\right) \label{equiv_from_minl}
	\end{equation}
and set 
	$
	f^{(2)}(s) = b(s) \left(\frac{d(\mu_1 \circ h^{-1})}{d\mu_2}(s)\right)^{1/\alpha} f^{(1)}\circ h^{-1}(s)
	$
	for all $s \in S_2$. Then $\left(f^{(2)}, \{\phi^{(2)}_t\}, \{c^{(2)}_t\}\right)$ is also a Rosi\'{n}ski triplet of $\{X_t\}$ on $(S_2, \mu_2)$. Moreover, if $\left(f^{(1)}, \{\phi^{(1)}_t\}, \{c^{(1)}_t\}\right)$ is minimal, then so is $\left(f^{(2)}, \{\phi^{(2)}_t\}, \{c^{(2)}_t\}\right)$. \\
	\noindent (b) Conversely, if $\left(f^{(i)}, \{\phi^{(i)}_t\}, \{c^{(i)}_t\}\right)$ is a minimal triplet of $\{X_t\}$ on $(S_i, \mu_i)$ for $i = 1,\, 2$, then \eqref{equiv_from_minl} holds. In particular, two group actions arising in two minimal representations of $\{X_t\}$ are conjugate. 
\end{thm}

\begin{proof} Note that (a) follows by mimicking the proof of Proposition~3.3 of \cite{rosinski:1995}. On the other hand, to establish (b), we have to immitate the proof of Theorem~3.6 in \cite{rosinski:1995} with a bit of care keeping in mind that $G$ is possibly non-abelian. More precisely, we have to expand $f^{(2)}_{\tau t}$ in two different ways (for each $\tau, t \in G$) and then invoke the uniqueness of $\Phi$ and $\eta$ in \eqref{reln_between_f_and_fast} (because of minimality) to carry out the argument in a similar fashion.
\end{proof}

The above result motivates the following notions, one of which will become an important invariant for a stationary \sas random field. 

\begin{definition} Suppose $\{X_t\}_{t \in G}$ is a stationary \sas random field and $\{\phi_t\}_{t \in G}$ is a nonsingular action (on a standard measure space $(S, \mu)$) arising in a Rosi\'{n}ski (resp. minimal) representation. Then the corresponding group measure space construction $\LLL^\infty(S, \mu) \rtimes G$ is called a Rosi\'{n}ski (resp. minimal) group measure space construction of $\{X_t\}_{t \in G}$. 
\end{definition}

The following result is an immediate consequence of Theorem~\ref{thm:minml_repn:conjugacy} above and should be regarded as one of the main results of this paper that builds a link between stationary \sas random fields and von Neumann algebras. 

\begin{thm} \label{thm:minml:invariant}
Let $\{X_t\}_{t \in G}$ be a stationary \sas random field indexed by a countable group $G$. Suppose, for each  $i=1,\,2$,  $\{\phi^{(i)}_t\}$ is a nonsingular $G$-action on a standard measure space $(S_i, \mu_i)$ arising in a minimal representation of $\{X_t\}$. Then 
\[
\LLL^\infty(S_1, \mu_1) \rtimes G \cong \LLL^\infty(S_2, \mu_2) \rtimes G
\]
as von Neumann algebras, i.e., $\{\phi^{(1)}_t\}$ and $\{\phi^{(2)}_t\}$ are $W^\ast$-equivalent. In other words, a stationary \sas random field (indexed by a countable group) has unique (up to von Neumann algebra isomorphism) minimal group measure space construction. 
\end{thm}

\begin{proof} By Theorem~\ref{thm:minml_repn:conjugacy}, $\{\phi^{(1)}_t\}$ and $\{\phi^{(2)}_t\}$ are conjugate $G$-actions. This means that there exists a bijection $h: S_1 \to S_2$ such that both $h$ and $h^{-1}$ are measurable, $\mu_1 \circ h^{-1} \sim \mu_2$, and for all $t \in G$, \eqref{eqn_conjugacy} holds. In particular, for all $t \in G$, \eqref{eqn_orbit_equiv} holds making $\{\phi^{(1)}_t\}_{t \in G}$ and $\{\phi^{(2)}_t\}_{t \in G}$ orbit equivalent. Since orbit equivalence implies $W^\ast$-equivalence (see, for instance, \cite{singer:1955}), it follows that as von Neumann algebras,
\[
\LLL^\infty(S_1, \mu_1) \rtimes G \cong \LLL^\infty(S_2, \mu_2) \rtimes G
\]
completing the proof. 	
\end{proof}

It transpires from Theorem~\ref{thm:minml:invariant} that \emph{the} minimal group measure space construction is an invariant for any stationary \sas random field indexed by a countable group. Of course, this invariant completely forgets the first and the last components of a minimal triplet  $\left(f, \{\phi_t\}, \{c_t\}\right)$ remembering only the middle one, namely, the nonsingular action $\{\phi_t\}$. However, it is clear from the references given in Section~\ref{sec:intro} that various probabilistic facets of stationary \sas random fields are closely tied up with ergodic theoretic properties of $\{\phi_t\}$ with $f$ and $\{c_t\}$ containing only little data on $\{X_t\}$. In light of this, the minimal group measure space construction, which encodes the ergodic theoretic properties of the action up to orbit equivalence, becomes a powerful invariant that carries a lot of information about diverse stochastic features of the random field. We will see manifestation of this heuristic in Theorems~\ref{thm:erg_charac}, \ref{thm:erg_part} and \ref{thm:Wstar_rigidity} and expect to discover more in this direction in future (see Section~\ref{sec:conj}). 

\subsection{Weak Mixing, Ergodicity and Complete Non-ergodicity} \label{sec:erg_wm_nonerg}

Recall that Rosi\'{n}ski representations do not enjoy the strong uniqueness property as the minimal ones. Therefore, two nonsingular actions arising in two different Rosi\'{n}ski representations (of the same stationary \sas field) may not be $W^\ast$-equivalent. However we do hope, in view of  \eqref{reln_between_f_and_fast}, that some operator algebraic properties will be preserved under further conditions on the group and/or the actions (see, e.g., Corollary~\ref{cor:rosinski_invariant}). With this goal in mind and with the intention of obtaining finer results, we specialize in the case when $G=\mathbb{Z}^d$ and the action $\{\phi_t\}_{t \in \mathbb{Z}^d}$ is free. 

The assumption of the underlying action being free is a technical one - it is important because it allows us to invoke  Theorem~\ref{thm:cntrl_n_erg_decomp} in the proof of Theorem~\ref{thm:erg_charac}, which is one of our main results. Keeping this in mind, we abuse the terminology a bit and introduce the following notion.

\begin{definition}
A Rosi\'{n}ski group measure space construction (or the corresponding Rosi\'{n}ski representation) is called free when the underlying nonsingular action is so. 
\end{definition}

\noindent By virtue of Theorem~\ref{thm:minml_repn:conjugacy} above, if the nonsingular action arising in a minimal representation of a stationary \sas random field is free then so is the action arising in another minimal representation of that field. In this situation, we will say that the minimal group measure space construction of the stationary \sas random field is free. This gives rise to the following definition with a further abuse of termonology. 
 
\begin{definition} \label{defn:SRF_erg_free}
We say that a stationary \sas random field is free when its minimal group measure space construction is so. 
\end{definition}

\noindent As Remark~\ref{remark:erg_free_not_restr} below suggests, all known families of stationary \sas random fields are free. 

Note that any (not necessarily stationary) \sas random field $\mathbf{X}=\{X_t\}_{t \in G}$ (defined on the probability space $(\Omega, \mathcal{A}, \mathbb{P})$ mentioned in the first paragraph of Section~\ref{sec:intro}) induces a canonical probability measure $\mathbb{P}_\mathbf{X}$ on the measurable space $\big(\prod_{t \in G} \mathbb{R},\, \bigotimes_{t \in G} \mathbb{B}_\mathbb{R}\big)$ defined by
\[
\mathbb{P}_\mathbf{X}:= \mathbb{P}\big(\{\omega \in \Omega: (X_t(\omega): t \in G) \in \cdot\}\big).
\]
It is easy to observe that $\{X_t\}_{t \in G}$ is (left) stationary if and only if the left tranlation $\{\tau_t\}_{t \in G}$ is a measure presevting $G$-action on the induced probability space $\big(\prod_{t \in G} \mathbb{R},\, \bigotimes_{t \in G} \mathbb{B}_\mathbb{R},\, \mathbb{P}_\mathbf{X}\big)$. 
\begin{definition} \label{defn:erg_SRF}
A stationary \sas random field $\mathbf{X}=\{X_t\}_{t \in G}$ is called ergodic if the measure-preserving $G$-action $\{\tau_t\}$ is so. 
\end{definition}
Ergodicity is an important notion for stochastic processes and it helps in establishing limit theorems for stationary \sas random fields as well. See, for example, \cite{mikosch:samorodnitsky:2000a}, whose main result (on estimation of ruin probabilities for stationary \sas random fields indexed by $\mathbb{Z}$; see Theorem~2.5 therein) is valid under the assumption of ergodicity. On the othar hand, if a stationary \sas random field $\{X_t\}_{t \in \mathbb{Z}^d}$ is ergodic, then  whenever $h: \mathbb{R} \to \mathbb{R}$ is a function such that $\mathbb{E}(|h(X_0)|) < \infty$, using multiparameter erdodic theorem of \cite{tempelman:1972} (see also Theorem~2.8 in Page 205 of \cite{krengel:1985}), we can find a strongly consistent estimator for $\theta:=\mathbb{E}(h(X_0))$ as follows:
\[
\hat{\theta}_n:= \frac{1}{(2n+1)^d} \sum_{\|t\|_\infty \leq n} h(X_t) \convas \theta
\]
as $n \to \infty$. Therefore finding characterizations of ergodicity is an important problem in statistics and probability theory. 

The multiparameter erdodic theorem mentioned in the above paragraph also yields that a measure-preserving $\mathbb{Z}^d$-action $\{\phi_t\}_{t \in \mathbb{Z}^d}$ on a probability space $(S, \mathcal{S}, \mu)$ is ergodic if and only if for all $A, B \in \mathcal{S}$,
\[
\lim_{n \to \infty} \frac{1}{(2n+1)^d} \sum_{\|t\|_\infty \leq n} \mu(A \cap \phi_t(B)) = \mu(A)\mu(B).
\]
Recall that such a $\mathbb{Z}^d$-action $\{\phi_t\}_{t \in \mathbb{Z}^d}$ is weakly mixing if for all $A, B \in \mathcal{S}$,
\[
\lim_{n \to \infty} \frac{1}{(2n+1)^d} \sum_{\|t\|_\infty \leq n} \big|\mu(A \cap \phi_t(B)) - \mu(A)\mu(B)\big| =0.
\]
\begin{definition}
A stationary \sas random field $\mathbf{X}=\{X_t\}_{t \in \mathbb{Z}^d}$ is called weakly mixing if the shift action $\{\tau_t\}_{t \in \mathbb{Z}^d}$ \big(on the on the induced probability space $\big(\prod_{t \in \mathbb{Z}^d} \mathbb{R},\, \bigotimes_{t \in \mathbb{Z}^d} \mathbb{B}_\mathbb{R},\, \mathbb{P}_\mathbf{X}\big)$\big) is so. 
\end{definition}
\noindent Weak mixing implies ergodicity but the converse does not hold in general. In case of $\mathbb{Z}^d$-indexed stationary \sas random fields, however, the two notions are equivalent; see Theorem~\ref{thm:erg_charac_neveu}, which gives another characterization of ergodicity of $\{X_t\}_{t \in \mathbb{Z}^d}$ (apart from being weakly mixing) and forms a building block in the proof of one of our main results. 

\begin{thm}[\cite{samorodnitsky:2005a} ($d=1$) and \cite{wang:roy:stoev:2013} ($d > 1$)] \label{thm:erg_charac_neveu}
	Let $\{X_t\}_{t \in \mathbb{Z}^d}$ be a stationary \sas random field. Then the following are equivalent.  
	\begin{enumerate}
		\item $\{X_t\}_{t \in \mathbb{Z}^d}$ is ergodic;
		\item $\{X_t\}_{t \in \mathbb{Z}^d}$ is weakly mixing;
		\item the nonsingular $\mathbb{Z}^d$-action $\{\phi_t\}_{t \in \mathbb{Z}^d}$ arising in some (equivalently, any) Rosi\'{n}ski representation is null, i.e., the positive part of $\{\phi_t\}$ has measure zero. 
	\end{enumerate}
\end{thm}

As mentioned before, with the help of Theorems~\ref{thm:cntrl_n_erg_decomp} and \ref{thm:erg_charac_neveu} above, we get out next result, which is a new characterization of ergodicity (equvalently weak mixing) of $\mathbb{Z}^d$-indexed stationary \sas random fields based on operator algebraic properties (namely, through Definition~\ref{no_II_1_factor}) of their Rosi\'nski group measure space constructions. 

\begin{thm} \label{thm:erg_charac}
Suppose that $\{X_t\}_{t \in \mathbb{Z}^d}$ is a stationary \sas random field and it has a free Rosi\'{n}ski group measure space construction. Then the following are equivalent.
\begin{enumerate}
	\item[(1)] $\{X_t\}_{t \in \mathbb{Z}^d}$ is ergodic;
	\item[(2)] $\{X_t\}_{t \in \mathbb{Z}^d}$ is weakly mixing;
	\item[(3)] the nonsingular $\mathbb{Z}^d$-action arising in some (equivalently, any) Rosi\'{n}ski representation is null;
	\item[(4)] some (equivalently, any) free Rosi\'{n}ski group measure space construction does not admit a $II_1$ factor in its central decomposition. 
\end{enumerate}
\end{thm}

\begin{proof}
  Because of Theorem~\ref{thm:erg_charac_neveu}, the statements (1), (2) and (3) are equivalent even without the hypothesis of existence of a free Rosi\'{n}ski group measure space construction. To complete the proof, it is enough verify that (3) and (4) are equivalent. To this end, we use the hypothesis of the theorem to get a free nonsingular $\mathbb{Z}^d$-action $\{\phi_t\}$ (on a standard measure space $(S, \mu)$) arising in a Rosi\'{n}ski representation of $\{X_t\}_{t \in \mathbb{Z}^d}$. We invoke Theorem~\ref{thm:erg_comp} to obtain the ergodic decomposition $\{(S_y, \mu_y): y \in Y\}$ of $S$ wrt the action $\{\phi_t \}_{t \in \mathbb{Z}^d}$. By virtue of Theorem~\ref{thm:cntrl_n_erg_decomp}, the free Rosi\'{n}ski group measure space construction $M = \LLL^\infty(S, \mu) \rtimes \mathbb{Z}^d$ has central decomposition 
  \begin{equation*}
  M = \int_Y M_y\, \nu(dy)  
  \end{equation*}
  with $M_y := \LLL^\infty(S_y, \mu_y) \rtimes \mathbb{Z}^d$ for each $y \in Y$. As per the convention mentioned in the beginning of Section~\ref{sec:erg_th}, all the statements used in the proof are true modulo null sets under the appropriate measure. \medskip
  
\noindent \textbf{Proof of (4) implies (3):} Suppose (4) holds and $\mathcal{P} \subseteq S$ is the positive part of $\{\phi_t\}_{t \in \mathbb{Z}^d}$. We have to show $\mu(\mathcal{P})=0$. Using the hypothesis (4), we can get a measurable $N_1 \subset Y$ such that $\nu(N_1)=0$ and for all $ y \in Z:=Y \setminus N_1$, the factor $M_y = \LLL^\infty(S_y, \mu_y) \rtimes \mathbb{Z}^d$ is not of type $II_1$. We first claim that it is enough to establish 
\begin{equation}
\mathcal{P} \subseteq \Psi^{-1}(N_1), \label{mathcal_P_subset_II_1}
\end{equation}
where $\Psi:S \to Y$ is the measurable map obtained in Theorem~\ref{thm:erg_comp}. To prove the claim, recall that $\nu=\mu \circ \Psi^{-1}$ and hence \eqref{mathcal_P_subset_II_1} yields 
\[
\mu(\mathcal{P})  \leq \mu \circ \Psi^{-1}(N_1) = \nu(N_1) =0.
\]

In order to complete the proof of (4) implies (3), it remains to show that \eqref{mathcal_P_subset_II_1} holds. Suppose that this inclusion does not hold. Then  
\[
\mathcal{P}_1 := \mathcal{P} \cap \Psi^{-1}(Z) \neq \emptyset
\] 
modulo $\mu$. This means by Theorem~\ref{thm:erg_comp},
\begin{align*}
0 < \mu(\mathcal{P}_1) 
&= \int_Y \mu_y(\mathcal{P}_1) d\nu(y)\\ 
&= \int_{Z} \mu_y(\mathcal{P}_1 \cap S_y) d\nu(y) + \int_{Y\setminus Z} \mu_y(\mathcal{P}_1 \cap S_y) d\nu(y)\\
&=\int_{Z} \mu_y(\mathcal{P}_1 \cap S_y) d\nu(y)\\ 
&= \int_{Z} \mu_y(S_y) d\nu(y).
\end{align*}
Therefore, there exists $z \in Z$ such that $\mu_{z}(S_{z})>0$. Since $\mathcal{P}$ and $S_{z}$ are both $\{\phi_t\}$-invariant, and $\{\phi_t\}$ restricted to $(S_{z}, \mu_{z})$ is ergodic, it follows that $S_{z} \subseteq \mathcal{P}$. 

Using the $\{\phi_t\}$-invariance of $\mathcal{P}$ once again, we get that $\{(S_y, \mu_y): y \in \Psi(\mathcal{P})\}$ is the ergodic decomposition of $\mathcal{P}$ wrt $\{\phi_t\}$ restricted to $\mathcal{P}$. Also $\mathcal{P}$, being the positive part, supports a $\{\phi_t\}$-invariant finite mesure $\tau \sim \mu$. Therefore by Theorem~\ref{thm:erg_comp}, $\tau_z:=\tau|_{S_z} \sim \mu_z$. Recall that $\mu_{z}(S_{z})>0$ and hence by equivalence, $\tau_z(S_{z})>0$. In particular, $S_z$ supports a nonzero $\{\phi_t\}$-invariant finite measure $\tau_z$ equivalent to $\mu_z$. This implies that $\{\phi_t\}$ restricted to $(S_z, \mu_z)$ is a positive action. Hence  Theorem~\ref{thm:factor_ergodic} yields that $M_{z} = \LLL^\infty(S_{z}, \mu_{z}) \rtimes \mathbb{Z}^d$ is a $II_1$ factor, which contradicts that $z \in Z$. Therefore \eqref{mathcal_P_subset_II_1} holds. This completes the proof of (4) implies (3) in Theorem~\ref{thm:erg_charac}. \medskip

\noindent \textbf{Proof of (3) implies (4):} Suppose (3) holds. We have to show that (4) holds. First note that
\[
0< \mu(S)=\int_{Y} \mu_y(S)d\nu(y)=\int_{Y} \mu_y(S_y)d\nu(y).
\]
This shows that there exits a measurable $N_0 \subset Y$ such that $\nu(N_0)=0$ and for all $y \in Y \setminus N_0$, we have  $\mu_y(S_y)>0$. Suppose (4) does not hold. Then there exists measurable $Y_0 \subseteq Y$ such that $\nu(Y_0)>0$ and for all $y \in Y_0$, the factor $M_y=\LLL^\infty(S_y, \mu_y) \rtimes \mathbb{Z}^d$ is of type $II_1$. Since $(Y, \nu)$ is a standard measure space, we may assume, without loss of generality, that $\nu(Y_0) \in (0, \infty)$. Set $V= Y_0 \cap N_0^c$. 

Clearly, by construction, $\nu(V)=\nu(Y_0) \in (0,\infty)$. Also for each $v \in V$, we get that  $\mu_v(S_v) >0$ and $M_v = \LLL^\infty(S_v, \mu_v) \rtimes \mathbb{Z}^d$ is a $II_1$ factor, and hence by Theorem~\ref{thm:factor_ergodic}, the action $\{\phi_t\}_{t \in \mathbb{Z}^d}$ restricted to $(S_v, \mu_v)$ is positive. This means that each $S_v$ supports a $\{\phi_t\}$-invariant finite measure $\tau_v \sim \mu_v$. Since $\mu_v(S_v) >0$, by equivalence, $\tau_v(S_v) >0$ for each $v \in V$. Therefore, by normalizing, we may assume that each $\tau_v$ is a probability measure on $S_v$. 

We will now arrive at the contradiction to (3) by constructing a measurable $\{\phi_t\}$-invariant set $\mathcal{Q} \subset S$ of strictly positive $\mu$ measure that supports a $\{\phi_t\}$-invariant finite measure $\tau \sim \mu|_{\mathcal{Q}}$. To this end, define $\mathcal{Q}= \Psi^{-1}(V)$, which is clearly $\{\phi_t\}$-invariant. Also, $$
\mu(\mathcal{Q}) = \mu \circ \Psi^{-1}(V) = \nu(V) >0.
$$
For all measurable $B \subseteq \mathcal{Q}$, define
\[
\tau(B):= \int_V \tau_v(B \cap S_v) d\nu(v).
\]

It is easy to check using monotone convergence theorem that $\tau$ is a measure on $\mathcal{Q}$ because each $\tau_v$ is a measure on $S_v$. Also 
\[
\tau(\mathcal{Q}) = \int_V \tau_v(\mathcal{Q} \cap S_v) d\nu(v) = \int_V \tau_v(S_v) d\nu(v) = \nu(V) \in (0, \infty)
\]
since each $\tau_v$ is a probability measure on $S_v$. This shows that $\tau$ is a nonzero finite measure on $\mathcal{Q}$. 

We now claim that $\{\phi_t\}_{t \in \mathbb{Z}^d}$ restricted to $\mathcal{Q}$ preserves the probability measure $\tau$. This is true because for each $t \in \mathbb{Z}^d$ and for each measurable $B \subseteq \mathcal{Q}$, we get by $\{\phi_t\}$-invariance of each $\tau_v$ and Theorem~\ref{thm:erg_comp} that 
\begin{align*}
\tau(\phi_t(B)) &= \int_V \tau_v(\phi_t(B) \cap S_v) d\nu(v)\\
                &= \int_V \tau_v(\phi_t(B \cap S_v)) d\nu(v)\\
                &= \int_V \tau_v(B \cap S_v) d\nu(v)\\
                &=\tau(B).
\end{align*}

Finally, we verify that $\tau \sim \mu|_{\mathcal{Q}}$. To achieve this, fix measurable $B \subseteq \mathcal{Q}$. Repeating an argument used in the proof of the reverse implication, we get
\begin{align*}
\mu(B) & =\int_Y \mu_y(B) d\nu(y) 
       = \int_V \mu_y(B \cap S_y) d\nu(y).
\end{align*}
Therefore $\mu(B)=\int_V \mu_y(B \cap S_y) d\nu(y)=0$ holds if and only if there exists a $\nu$-null set $N_B$  such that for each $y \in V \setminus N_B$,  $\mu_y(B \cap S_y)=0$ holds if and only if there exists a $\nu$-null set $N_B$ such that for each $y \in V \setminus N_B$, $\tau_y(B \cap S_y)=0$ holds if and only if $\tau(B)=\int_V \tau_y(B \cap S_y) d\nu(y)=0$ holds. This argument is true for any measurable subset $B$ of  $\mathcal{Q}$ and it shows $\tau \sim \mu|_{\mathcal{Q}}$. 

Therefore we have produced a measurable $\{\phi_t\}$-invariant subset $\mathcal{Q}$ of $S$ such that $\mu(\mathcal{Q}) >0$ and $\mathcal{Q}$ supports a nonzero finite measure $\tau~(\sim \mu|_{\mathcal{Q}})$ that preserves $\{\phi_t\}$. This shows that $\{\phi_t\}$ has a nontrivial positive part contradicting (3). Therefore (4) must hold. This finishes the proof of Theorem~\ref{thm:erg_charac}. 
\end{proof}

\begin{remark}\label{remark:positive_part_using_II_1}
\noindent \textnormal{In the above proof, we roughly establish that} 
\begin{equation*}
\mbox{``}\mathcal{P}=\Psi^{-1}\big(\{y \in Y: M_y=\LLL^\infty(S_y, \mu_y) \rtimes \mathbb{Z}^d \mbox{ is a $II_1$ factor}\}\big)\mbox{''}
\end{equation*}
\textnormal{except that the set $\{y \in Y: M_y \mbox{ is a $II_1$ factor}\}$ may not be a measurable subset of $Y$. We resolve this measurability issue with the help of various facets of the ergodic decomposition as given in Theorem~\ref{thm:erg_comp}. In fact, the same proof will establish the equivalence of the statements (3) and (4) in Theorem~\ref{thm:erg_charac} for any countable indexing group, not just $\mathbb{Z}^d$. Therefore, if Theorem~\ref{thm:erg_charac_neveu} holds for stationary \sas fields indexed by more general class of groups, then so does Theorem~\ref{thm:erg_charac}. In particular, the obstracle is not von Neumann algebraic but ergodic theoretic, namely, the unavailibility of ergodic theorem for nonsingular actions of more general groups; see the discussions in \cite{jarrett:2019} and also in Section~\ref{sec:conj} below.} 
\end{remark}

\begin{remark} \label{remark:erg_free_not_restr}
\textnormal{The assumption of existence of a free Rosi\'{n}ski group measure space construction is neither needed for the equivalence of the first three statements of Theorem~\ref{thm:erg_charac} (as seen in Theorem~\ref{thm:erg_charac_neveu}) nor very restrictive because all known families of stationary \sas random fields are generated by free actions.} 
\end{remark}

We will say that a stationary \sas random field $\{X_t\}_{t \in \mathbb{Z}^d}$ has a nontrivial ergodic part if the stationary \sas random field $\{X^N_t\}_{t \in \mathbb{Z}^d}$ in (3.6) of \cite{wang:roy:stoev:2013} is nontrivial. Otherwise, we will say that $\{X_t\}_{t \in \mathbb{Z}^d}$ does not have a nontrivial ergodic part and this should be thought of as ``complete non-ergodicity'' of $\{X_t\}$. The following result also uses Definition~\ref{no_II_1_factor} and it is essentially a characterization of absolute lack of ergodicity for a stationary \sas random field.

\begin{thm} \label{thm:erg_part}
	Suppose that $\{X_t\}_{t \in \mathbb{Z}^d}$ is a stationary \sas random field and it has a free Rosi\'{n}ski group measure space construction. Then the following are equivalent.
	\begin{enumerate}
		\item[(i)] $\{X_t\}_{t \in \mathbb{Z}^d}$ does not have a nontrivial ergodic part;
		\item[(ii)] the nonsingular $\mathbb{Z}^d$-action $\{\phi_t\}$ arising in some (equivalently, any) Rosi\'{n}ski representation is positive, i.e., the null part of $\{\phi_t\}$ is of measure zero;
		\item[(iii)] some (equivalently, any) free Rosi\'{n}ski group measure space construction $M$ admits only $II_1$ factor(s) in its central decomposition.
	\end{enumerate}
\end{thm}

\begin{proof} Clearly, (i) and (ii) are equivalent thanks to Theorem~\ref{thm:erg_charac_neveu}. We will establish the equivalence of (ii) and (iii) following the notation and terminology used in the proof of Theorem~\ref{thm:erg_charac}. Fix a free action $\{\phi_t\}$ arising in a Rosi\'nski representation of $\{X_t\}$.
	
First assume that (ii) holds but (iii) does not. Then there exists $Y_0 \subset Y$ such that $\nu(Y_0)>0$ and for all $y \in Y_0$, the factor $M_y$ is not of type $II_1$. Then it is possible to check, following the arguments used in the proof of Theorem~\ref{thm:erg_charac}, that $\Psi^{-1}(Y_0) \subseteq \mathcal{N}$, the null part of $\{\phi_t\}$. This is a contradiction to (ii) because 
$$
\mu(\mathcal{N}) \; \geq \; \nu \circ\Psi^{-1}(Y_0) \; = \; \nu(Y_0)>0.
$$
This completes the proof of (ii) implies (iii).

Now suppose that (iii) holds but (ii) does not. This means $\mu(\mathcal{N})>0$. Then $Y_1 = \Psi(\mathcal{N})$ satisfies
\[
\nu(Y_1) \; = \; \mu \circ \Psi^{-1}(Y_1) \; = \; \mu\big(\Psi^{-1}(\Psi(\mathcal{N}))\big) \; \geq \; \mu(\mathcal{N})) \; > \; 0
\]
and an argument similar to the one used in the proof of Theorem~\ref{thm:erg_charac} yields that for almost all $y \in Y_1$, the factor $M_y$ is not of type $II_1$, which contradicts the hypothesis that (iii) holds. This completes the proof of (iii) implies (ii) and hence of Theorem~\ref{thm:erg_part}. 
\end{proof}

\begin{remark} \label{remark:LRD} (Connection to Long Range Dependence) \textnormal{In Theorems~\ref{thm:erg_charac} and \ref{thm:erg_part} above, we notice two diametrically opposite situations for stationary \sas random fields indexed by $\mathbb{Z}^d$. No factor in the central decomposition of some (equivalently, any) free Rosi\'nski group measure space construction is of type $II_1$ (almost surely) if and only if the random field is ergodic. On the other hand, almost all such factors are of type $II_1$ if and only if the random field is fully non-ergodic. Since ergodicity is a form of asymptotic independence, we can come up with the following heuristic. In this setup, having more $II_1$ factors in the central decomposition of a Rosi\'nski group measure space construction should be thought of as presence of longer memory (i.e., stronger dependence) in the stationary \sas random field. This goes very well with the spirit of \cite{samorodnitsky:2004a}, who initiated this connection between length of memory of stable fields and ergodic theoretic properties of the underlysing nonsingular action. We take this connection forward to the realm of operator algebras as well.}
\end{remark}

The following corollaries follow immediately from  Theorems~\ref{thm:erg_charac} and \ref{thm:erg_part}, and hence their proofs are omitted. 

\begin{cor} \label{cor:Y_ctble}
With the assumptions of Theorem~\ref{thm:erg_charac} and the notation used in its proof, if $Y$ is countable and $\nu$ is the counting measure on $Y$, then $\{X_t\}_{t \in \mathbb{Z}^d}$ is ergodic (equivalently, weakly mixing) if and only if for each $y \in Y$, the factor $\LLL^\infty(S_y, \mu_y) \rtimes G$ is not of type $II_1$. ALso, $\{X_t\}_{t \in \mathbb{Z}^d}$ does not have a nontrivial ergodic part if and only if each factor $\LLL^\infty(S_y, \mu_y) \rtimes G$ is of type $II_1$.
\end{cor}

\begin{cor} \label{cor:rosinski_invariant}
Not admitting a $II_1$ factor (and also, admitting only $II_1$ factor(s)) in its central decomposition is an invariant for any free Rosi\'{n}ski (not necessarily minimal) group measure space construction of a fixed stationary \sas random field indexed by $\mathbb{Z}^d$. 
\end{cor}

Corollary~\ref{cor:Y_ctble} is a useful special case of Theorem~\ref{thm:erg_charac} and it will be applied in the proof of Theorem~\ref{thm:ex_MC} below. On the other hand, Corollary~\ref{cor:rosinski_invariant} gives us hope that a Rosi\'{n}ski group measure space construction too has reasonable information about a $\mathbb{Z}^d$-indexed  stationary \sas random field provided, of course, it is free. 

\subsection{$W^\ast_m$- and $W^\ast_R$-rigidities} \label{sec:Wstar_sas}

In light of the results presented in the previous subsection and inspired by the concept of $W^\ast$-superrigidity coined by Sorin Popa (see the discussions in Section~\ref{sec:Wstar_super} and also the survey of \cite{ioana:2018}), we introduce a bunch of new notions in this paper. We expect to use them in our future work as well. With this in mind, we define two equivalence relations for stationary \sas random fields when the indexing group is allowed to be any countable group, not just $\mathbb{Z}^d$. 
 
\begin{definition} \label{defn:Wstar_equiv_SRF}
Two stationary \sas fields $\big\{X^{(1)}_t\big\}_{t \in G_1}$ and $\big\{X^{(2)}_t\big\}_{t \in G_2}$ indexed by two (possibly different) countable groups are called $W^\ast_R$-equivalent (resp., $W^\ast_m$-equivalent) if a free  Rosi\'{n}ski (resp., minimal) group measure space construction of  $\big\{X^{(1)}_t\big\}_{t \in G_1}$ is isomorphic (as von Neumann algebra) to a free  Rosi\'{n}ski (resp., minimal) group measure space construction of  $\big\{X^{(2)}_t\big\}_{t \in G_2}$. 
\end{definition}

In particular, the notion of $W^\ast_R$-equivalence needs the existence of free  Rosi\'{n}ski group measure space constructions for both stationary \sas random fields. On the other hand, $W^\ast_m$-equivalence requires that both random fields are free is the sense of Definition~\ref{defn:SRF_erg_free}. In view of Remark~\ref{remark:erg_free_not_restr} above, these requirements do not constrain the notions too much. Since by Theorem~3.1 of \cite{rosinski:1995}, any minimal representation is a Rosi\'{n}ski representation (but the converse does not hold), it follows that
\begin{equation}
W^\ast_m\mbox{-equivalence} \;\; \Rightarrow \;\; W^\ast_R\mbox{-equivalence}  \label{impli_wstarm_wstarR}
\end{equation}
even though we are not sure of the converse. While it seems plausible that $W^\ast_R$-equivalence may not imply $W^\ast_m$-equivalence, we do not know of any counter-example to confirm this belief. Nor can we show that these two notions are equivalent under additional conditions on the groups and/or their actions and/or the fields. 

Using the above two equivalence relations, we introduce two notions of $W^\ast$ rigidity for a property of stationary \sas random fields as follows. 

\begin{definition}
A property $P$ of stationary \sas random fields (indexed by a class $\mathcal{G}$ of countable groups) is called $W^\ast_R$-rigid (resp., $W^\ast_m$-rigid) for $\mathcal{G}$ if whenever two such fields (not necessarily indexed by the same group) are $W^\ast_R$-equivalent (resp., $W^\ast_m$-equivalent), one enjoys property $P$ if and only if the other one does. 
\end{definition}

The following fact follows immediately from \eqref{impli_wstarm_wstarR} once we understand the terminology. 
	
\begin{thm} \label{thm:impli_Wstar_rigidities_SRF}
If a property $P$ of stationary \sas random fields is $W^\ast_R$-rigid for a class $\mathcal{G}$ of countable groups, then it is also $W^\ast_m$-rigid for $\mathcal{G}$.
\end{thm}

\begin{proof}
Suppose $P$ is a $W^\ast_R$-rigid property. We have to show that $P$ is $W^\ast_m$-rigid. Take $G_1, G_2 \in \mathcal{G}$, two free stationary \sas random fields $\big\{X^{(1)}_t\big\}_{t \in G_1}$ and $\big\{X^{(2)}_t\big\}_{t \in G_2}$ such that they are $W^\ast_m$-equivalent, and $\big\{X^{(1)}_t\big\}_{t \in G_1}$ has property $P$. Then by \eqref{impli_wstarm_wstarR}, they are also $W^\ast_R$-equivalent. Hence by our hypothesis (of $P$ being $W^\ast_R$-rigid), $\big\{X^{(2)}_t\big\}_{t \in G_2}$ also has property $P$. By symmetry of the argument, it follows that $P$ is $W^\ast_m$-rigid. 
\end{proof}

Theorem~\ref{thm:impli_Wstar_rigidities_SRF} can be described pictorially with the help of Figure~\ref{fig:Wstar_reln}. As the figure suggests, more properties of stationary \sas random fields are $W^\ast_m$-rigid. This is because Theorem~\ref{thm:minml:invariant} above makes the minimal group measure space construction a powerful invariant that perhaps remembers a lot of features of a stationary \sas field indexed by any countable group. 

\begin{figure}[h]
	\centering
	\includegraphics{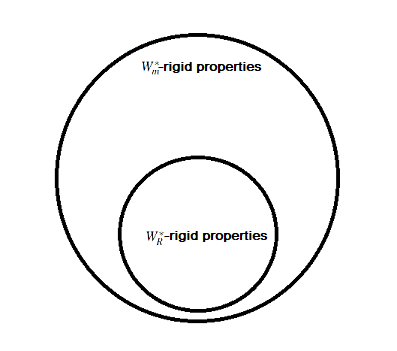}
	\caption{The relation between $W^\ast$-rigidities}
	\label{fig:Wstar_reln}
\end{figure}

On the other hand, thanks to \eqref{reln_between_f_and_fast}, we do expect any free Rosi\'{n}ski group measure space construction to remember many properties of a stationary \sas random field as well making the two circles in Figure~\ref{fig:Wstar_reln} come close to each other. In other words, our eventual aim should be to reverse the implication in \eqref{impli_wstarm_wstarR} (and hence in Theorem~\ref{thm:impli_Wstar_rigidities_SRF} for many properties) under extra algebraic restriction(s) on the group and/or ergodic theoretic condition(s) on the action and/or ptobabilistic constraint(s) on the random field. At this point, we are far  from this goal with the only partial success story being the following result.

\begin{thm} \label{thm:Wstar_rigidity}
	Ergodicity (equivalently, weak mixing) is a $W^\ast_R$-rigid (and hence a  $W^\ast_m$-rigid) property for $\mathcal{G}:=\big\{\mathbb{Z}^d: d \in \mathbb{N}\big\}$. Also, having a nontrivial ergodic part is a $W^\ast_R$-rigid (and hence a  $W^\ast_m$-rigid) property for $\mathcal{G}$. 
\end{thm} 

\begin{proof}
Take two stationary \sas random fields $\big\{X^{(1)}_t\big\}_{t \in \mathbb{Z}^{d_1}}$ and $\big\{X^{(2)}_t\big\}_{t \in \mathbb{Z}^{d_2}}$ such that they are $W^\ast_R$-equivalent. This means that for each $i=1,\,2$, there exists a free nonsingular action $\{\phi^{(i)}_t\}_{t \in \mathbb{Z}^{d_i}}$ on a standard Borel space $(S_i, \mu_i)$ arising in a Rosi\'{n}ski representation such that 
\[
\LLL^\infty(S_1, \mu_1) \rtimes \mathbb{Z}^{d_1} \cong \LLL^\infty(S_2, \mu_2) \rtimes \mathbb{Z}^{d_2}. 
\]
as von Neumann algebras. 

In particular, $\LLL^\infty(S_1, \mu_1) \rtimes \mathbb{Z}^{d_1}$ does not admit a $II_1$ factor in its central decomposition if and only if same is true about $\LLL^\infty(S_2, \mu_2) \rtimes \mathbb{Z}^{d_2}$. Therefore, by Theorem~\ref{thm:erg_charac}, $\big\{X^{(1)}_t\big\}_{t \in \mathbb{Z}^{d_1}}$ is ergodic (equiv., weakly mixing) if and only if $\big\{X^{(2)}_t\big\}_{t \in \mathbb{Z}^{d_2}}$ is so. This completes the proof of the first statement, i.e., ergodicity (equiv, weak mixing) is a $W^\ast_R$-rigid (and hence a  $W^\ast_m$-rigid) property for $\mathcal{G}$. The second statement follows similarly using Theorem~\ref{thm:erg_part}. 
\end{proof}

It follows from Theorem~\ref{thm:Wstar_rigidity} that ergodicity (and also having a nontrivial ergodic part) is an orbit equivalence rigid property for stationary stable random fields indexed by  $\mathbb{Z}^d$, $d \in \mathbb{N}$ in the following sense. 

\begin{cor} \label{cor:OE_rigidity}
	If two stationary \sas fields $\big\{X^{(1)}_t\big\}_{t \in \mathbb{Z}^{d_1}}$ and $\big\{X^{(2)}_t\big\}_{t \in \mathbb{Z}^{d_2}}$ are generated by orbit equivalent free nonsingular actions (of $\mathbb{Z}^{d_1}$ and $\mathbb{Z}^{d_1}$, respectively), then the following statements hold.
	
	\noindent (a)~$\big\{X^{(1)}_t\big\}_{t \in \mathbb{Z}^{d_1}}$ is ergodic (equiv. weakly mixing) if and only if $\big\{X^{(2)}_t\big\}_{t \in \mathbb{Z}^{d_2}}$ is so. 
	
	\noindent (b)~$\big\{X^{(1)}_t\big\}_{t \in \mathbb{Z}^{d_1}}$ has a nontrivial ergodic part if and only if so does $\big\{X^{(2)}_t\big\}_{t \in \mathbb{Z}^{d_2}}$. 
\end{cor}

\begin{proof} Under the hypothesis of Corollary~\ref{cor:OE_rigidity}, it follows using \eqref{impli_conj_oe_wstar} that the corresponding free Rosi\'{n}ski group measure space constructions are isomorphic as von Neumann algebras. In other words, $\big\{X^{(1)}_t\big\}_{t \in \mathbb{Z}^{d_1}}$ and $\big\{X^{(2)}_t\big\}_{t \in \mathbb{Z}^{d_2}}$ are $W^\ast_R$-equivalent. Hence, by Theorem~\ref{thm:Wstar_rigidity}, $\big\{X^{(1)}_t\big\}_{t \in \mathbb{Z}^{d_1}}$ is ergodic (equivalently, weakly mixing) if and only if $\big\{X^{(2)}_t\big\}_{t \in \mathbb{Z}^{d_2}}$ is so. This establishes (a). Statement~(b) follows similarly. 
\end{proof}

Note that without Corollary~\ref{cor:OE_rigidity}, it would not at all be obvious that orbit equivalence preserves ergodicity (or the complete lack of it) in our setup. We would also like to mention that the indexing groups having possibly different ranks (as $\mathbb{Z}$-modules) is actually very useful in the context of orbit equivalence. This is due to the seminal result of \cite{connes:feldman:weiss:1981} which states that any nonsingular action of $\mathbb{Z}^d$ (more generally, of any amenable group) is orbit equivalent to a nonsingular $\mathbb{Z}$-action. Therefore, it is now possible to associate a stationary \sas process to any stationary \sas random field indexed by $\mathbb{Z}^d$ in an ergodicity-preserving manner. 

\subsection{Stationary Max-stable Random Fields} \label{sec:max_stable}
\textnormal{All our results presented in Section~\ref{sec:main_results} have analogues for stationary max-stable (more specifically, $\alpha$-Fr\'{e}chet) random fields as well. Such random fields enjoy a spectral representation similar to \eqref{repn:rosinski_intro} except that the cocycle is trivial, the \sas random measure $M$ is replaced by an independently scattered $\alpha$-Fr\'{e}chet random sup-measure with control measure $\mu$, and the integral in \eqref{repn:rosinski_intro} becomes an extremal integral; see \cite{stoev:taqqu:2005} and \cite{wang:stoev:2010b} for the details in the $G=\mathbb{Z}$ case. The extension to any countable indexing group $G$ is trivial.}

\textnormal{There is a parallel notion of minimality of spectral representations for max-stable fields; see, for example, \cite{wang:stoev:2009}, Proposition~6.1 of which states that nonsingular actions arising in two minimal spectral representations of a $\mathbb{Z}$-indexed stationary $\alpha$-Fr\'{e}chet random field are conjugate. The same arguments holds for a general coutable indexing group $G$ (with a bit of care keeping in mind the potential noncommutativity of the group as in the proof of Theorem~\ref{thm:minml_repn:conjugacy}). Therefore the minimal group measure space construction becomes an invariant in the max-stable case as well.}
	
\textnormal{Using Theorem~5.3 of \cite{wang:roy:stoev:2013}, we can establish an analogue of Theorem~\ref{thm:erg_charac} for stationary $\alpha$-Fr\'{e}chet random fields by following the same proof. This immediately gives rise to a similar $W^\ast$-rigidity theory in the max-stable world establishing parallels of all results proved in Section~\ref{sec:main_results}. This actually goes very well with the spirit of \cite{wang:stoev:2010a}, who established a strong association between sum-stable and max-stable paradigms.}

\section{Applications and Future Directions} \label{sec:final_remarks}

In this section, we view some of the known classes of examples through the lens of our results, and finally present a few open problems and conjectures. In particular, we discuss how the indexing group in some of our results can perhaps be extended to more general groups (as opposed to only $\mathbb{Z}^d$).

\subsection{Examples} \label{sec:example} This section is devoted to applications of our results to a few examples as described below. They will help us understand the notation and terminology used in this work. For simplicity, all the cocycles in the Rosi\'nski representations of the stationary \sas fields discussed in this subsection will be taken to be the trivial one (i.e., $c_t \equiv 1$ for all $t$). We would also like to mention that all these examples have canonical max-stable (more precisely, $\alpha$-Fr\'echet)  analogues and the conclusions drawn about them will be the same due to the discussions in Section~\ref{sec:max_stable}. 

\begin{example}[Stationary Mixed Moving Average \sas Random Field] \textnormal{This $\mathbb{Z}^d$-indexed random field was introduced (in the $d=1$ case) by \cite{surgailis:rosinski:mandrekar:cambanis:1993}. Suppose $(Y, \nu)$ is a $\sigma$-finite standard measure space, $\zeta$ is the counting measure on $\mathbb{Z}^d$, $M$ is an \sas random measure on $S= Y \times \mathbb{Z}^d$ with control measure $\mu = \nu \otimes \zeta$, and $f \in \LLL^\alpha({Y \times \mathbb{Z}^d}, \nu \otimes \zeta)$ is a real-valued function. Stationary mixed moving average \sas random field is defined by}
\[
X_t = \int_{Y \times \mathbb{Z}^d} f(y, z+t) dM(y,z), \;\;\; t \in \mathbb{Z}^d. 
\]	

\textnormal{In the above Rosi\'nski representation, the measure-preserving (and hence nonsingular)  $\mathbb{Z}^d$-action on $S= Y \times \mathbb{Z}^d$ is given by}
\[
\phi_t(y,z)=(y,z+t), \;\;\;\;\; (y,z) \in Y \times \mathbb{Z}^d
\]
\textnormal{for all $t \in \mathbb{Z}^d$. Note that using Fubini's theorem the measure $\mu = \nu \otimes \zeta$ can be rewritten as}
\[
\mu(A) = \int_Y \mu_y(A) d\nu(y), \;\;\;\; A \subseteq S,
\]
\textnormal{where $\mu_y(A) = \zeta(\{z \in \mathbb{Z}^d: (y, z) \in A\})$ for all $y \in Y$. Each $\mu_y$ is supported on $S_y:=\{y\} \times \mathbb{Z}^d$ giving rise to the ergodic decomposition $\{(S_y, \mu_y): y \in Y\}$ of $\{\phi_t\}$.} 

\textnormal{The action $\{\phi_t\}_{t \in \mathbb{Z}^d}$ restricted to each $(S_y, \mu_y)$ is just a shift action on the second component, and is free and ergodic. Moreover, it is well-known that such a shift action is of Krieger type $II_\infty$ (this actually follows from the uniqueness (up to a constant) of Haar measure on $\mathbb{Z}^d$). In particular, the corresponding free Rosi\'nski group measure space construction
$
\LLL^\infty({Y \times \mathbb{Z}^d}, \nu \otimes \zeta) \rtimes \mathbb{Z}^{d}
$
does not admit a $II_1$ factor (in fact, it admits only $II_\infty$ factors; see, for example, \cite{jones:2009} for the definition of a $II_\infty$ factor) in its central decomposition}
\[
\LLL^\infty({Y \times \mathbb{Z}^d}, \nu \otimes \zeta) \rtimes \mathbb{Z}^{d} = \int_Y \big(\LLL^\infty({\{y\} \times \mathbb{Z}^d}, \, \mu_y) \rtimes \mathbb{Z}^{d}\big) \, d\nu(y)
\]
\textnormal{obtained via Theorem~\ref{thm:cntrl_n_erg_decomp}. Therefore, by Theorem~\ref{thm:erg_charac}, $\{X_t\}_{t \in \mathbb{Z}^d}$ is ergodic.}
\end{example}

\begin{example}[Stationary Sub-Gaussian \sas Random Field] \textnormal{This is a slight generalization of Example~6.1 of \cite{roy:samorodnitsky:2008}. Suppose $\mu$ is a probability measure on $S=\mathbb{R}^{\mathbb{Z}^d}$ such that the coordinate projections $\{p_t\}_{t \in \mathbb{Z}^d}$ is an ergodic stationary random field under $\mu$. Take an \sas random measure $M$ on $\mathbb{R}^{\mathbb{Z}^d}$ with control measure $\mu$ and define a stationary \sas random field}
\begin{equation}
X_t = \int_{\mathbb{R}^{\mathbb{Z}^d}} p_t dM, \;\;\;\; t \in \mathbb{Z}^d.  \label{int_repn:sub_Gaussian}
\end{equation}
\textnormal{When $\{p_t\}_{t \in \mathbb{Z}^d}$ is a sequence of i.i.d.\ Gaussuan random variables, then  $\{X_t\}_{t \in \mathbb{Z}^d}$  becomes a stationary sub-Gaussian \sas field; see Example~5.1 in \cite{samorodnitsky:2004a} and also Example~6.1 in \cite{roy:samorodnitsky:2008}.} 

\textnormal{In the Rosi\'nski representation \eqref{int_repn:sub_Gaussian}, the underlying action is simply the shift action of $\mathbb{Z}^d$ on $\mathbb{R}^{\mathbb{Z}^d}$, and it is p.m.p.\ and ergodic because of our assumptions on $\mu$. In particular, by Theorem~\ref{thm:factor_ergodic}, the corresponding Rosi\'nski group measure space construction}
\[
M=\LLL^\infty\big(\mathbb{R}^{\mathbb{Z}^d}, \mu\big) \rtimes \mathbb{Z}^{d}
\]
\textnormal{is itself a factor, which is of type $II_1$. This means that the set $Y$ in the central decomposition \eqref{decomp:central} of this crossed product von Neumann algebra $M$ is a singleton set with $\nu$ being the counting measure on it. Hence, another application of Theorem~\ref{thm:erg_charac} yields that  $\{X_t\}_{t \in \mathbb{Z}^d}$ is not ergodic. Nor does it have a nontrivial ergodic part.}
\end{example}

\begin{example}[Stationary \sas Process Associated with a Markov Chain] \label{ex:sas_MC} \textnormal{This class of examples was introduced by \cite{rosinski:samorodnitsky:1996} (see also Example~4.1 in \cite{samorodnitsky:2005a}). Following them, we would assume that $G=\mathbb{Z}$. It is possible to extend this example to a $\mathbb{Z}^d$-indexed one in parallel to Example~6.2 of \cite{wang:roy:stoev:2013} but it would complicate the notation a lot. Therefore, we refrain from doing so.}
	
\textnormal{We start with a Markov chain on a countable state space $C$ that decomposes into nonempty communication classes $\{C_i: i \in I\}$. Clearly, $I$ is also countable and hence we assume that $I \subseteq \mathbb{N}$. For each $i \in I$, fix a state $l_i \in C_i$. Assume that for each $i \in I$, the Markov chain restricted to $C_i$ is aperiodic and recurrent (and of course, irreducible) with the transition probability matrix $\big(p^{(i)}_{jk}\big)_{j, k \in C_i}$, and the invariant measure $\pi^{(i)} = (\pi^{(i)}_l: l \in C_i)$ on $C_i$ with $\pi^{(i)}_l \in (0, \infty)$ for each $l \in C_i$ and $\pi^{(i)}_{l_i}=1$ (such an invariant measure exists uniquely because of irreducibility and recurrence of the restricted Markov chain; see, e.g., Proposition~2.12.3 of \cite{resnick:1992}). Let $S= \cup_{i \in I} S_i$, where $S_i = C_i^\mathbb{Z}$.} 

\textnormal{Fix $i \in I$ for now. For each $l \in C_i$, let $\mathbb{P}^{(i)}_l$ be the unique probability measure on 
$$
S_i = C_i^\mathbb{Z}=\{x=(x(t): t \in \mathbb{Z}):\mbox{ each }x(t) \in C_i\}
$$ 
such that under $\mathbb{P}^{(i)}_l$, $x(0)=l$, $(x(0), x(1), x(2), \ldots)$ is a Markov chain with transition probability matrix $\big(p^{(i)}_{jk}\big)_{j, k \in C_i}$, and $(x(0), x(-1), x(-2), \ldots)$ is a Markov chain with transition probability matrix $\big(\pi^{(i)}_k p^{(i)}_{kj}/\pi^{(i)}_j \big)_{j, k \in C_i}$\,. Define a $\sigma$-finite measure $\mu_i$ on $S_i = C_i^\mathbb{Z}$ as}
\begin{equation}
\mu_i:=\sum_{l \in C_i} \pi^{(i)}_l \mathbb{P}^{(i)}_l. \label{defn_mu_i}
\end{equation}
\textnormal{Clearly, $\mu_i$ is a finite measure if and only if $\pi^{(i)}$ is so, which is equivalent to saying the underlying Markov chain restricted to $C_i$ is positive recurrent.} 	

\textnormal{Now we vary $i \in I$ and assign this measure $\mu_i$ defined by \eqref{defn_mu_i} on each $S_i = C_i^\mathbb{Z}$ to obtain a measure $\mu$ on $S= \cup_{i \in I} S_i$. In other words, $\mu$ is the unique $\sigma$-finite measure on $S= \cup_{i \in I} S_i$ such that its restriction to each $S_i$ is $\mu_i$. Let $\{\phi_t\}_{t \in \mathbb{Z}^d}$ be the shift action on $S$. This means that for all $t \in \mathbb{Z}^d$,}
\[
\phi_t(x)(u)=x(u+t), \;\;\;\; x \in S, \, u \in \mathbb{Z}. 
\]
\textnormal{Clearly this $\mathbb{Z}^d$ action is free.}
	
\textnormal{Because of the invariance of the measures  $\pi^{(i)}$, $i \in I$, the action  $\{\phi_t\}_{t \in \mathbb{Z}^d}$ preserves the measure $\mu$. Recall that for each $i \in I$, we have fixed $l_i \in C_i$ and chosen the invariant measure $\pi^{(i)}$ such a way that $\pi^{(i)}_{l_i}=1$. Hence the map $f: S \to \mathbb{R}$ defined by}
\[
f(x) := \sum_{i \in I} {2^{-i/\alpha}}\,\mathbbm{1}_{(x \in S_i,\, x(0)=l_i)} 
\]
\textnormal{belongs to $\LLL^\alpha(S, \mu)$. Suppose $M$ is an \sas random measure on $S$ with control measure $\mu$. Define a $\mathbb{Z}$-indexed stationary \sas random field using the Rosi\'nski representation}
\begin{equation}
X_t = \int_S f \circ \phi_t(x) dM(x), \;\;\;\; t \in \mathbb{Z}. \label{int_repn:MC_ex}
\end{equation}
\textnormal{The following result gives a criterion for ergodicity (equivalently, weak mixing) as well as for the complete absence of it for the random field defined above.}
\end{example}

\begin{thm} \label{thm:ex_MC} The random field $\{X_t\}_{t \in \mathbb{Z}}$ defined by \eqref{int_repn:MC_ex} is ergodic if and only if all the states of the underlying Markov chain are null recurrent. On the other hand, $\{X_t\}_{t \in \mathbb{Z}}$ does not have a nontrivial ergodic part if and only if all the states are positive recurrent. 
\end{thm}

\begin{proof} \textnormal{Firstly observe that each $S_i$ is a $\{\phi_t\}$-invariant subset of $S$ and $\mu = \sum_{i \in I} \mu_i$. The underlying Markov chain being irreducible and recurrent on each communication class $C_i$, the shift action $\{\phi_t\}$ restricted to each $S_i$ is ergodic. Therefore $\{(S_i, \mu_i): i \in I\}$ is the ergodic decomposition of $\{\phi_t\}_{t \in \mathbb{Z}^d}$ with $Y=I$ and $\nu$ being the counting measure on $I$.}
	
\textnormal{Clearly, the Rosi\'nski group measure space construction corresponding to the intergral representation \eqref{int_repn:MC_ex}}
	\[
	\LLL^\alpha(S, \mu) \rtimes \mathbb{Z}  = \LLL^\alpha\big(\cup_{i \in I} C_i^\mathbb{Z}, \mu\big) \rtimes \mathbb{Z}
	\]
is free. By Theorem~\ref{thm:cntrl_n_erg_decomp}, it has the central decomposition
	\[
	\LLL^\alpha(S, \mu) \rtimes \mathbb{Z} \cong \bigoplus_{i \in I} \big(\LLL^\alpha(S_i, \mu_i) \rtimes \mathbb{Z}\big) = \bigoplus_{i \in I} \big(\LLL^\alpha(C_i^\mathbb{Z}, \mu_i) \rtimes \mathbb{Z}\big)\,.
	\]
	
\textnormal{Using Theorem~\ref{thm:factor_ergodic} above and the discussions in Example~4.1 of \cite{samorodnitsky:2005a}, it follows that the factor $\big(\LLL^\alpha(C_i^\mathbb{Z}, \mu_i) \rtimes \mathbb{Z}\big)$ is of type $II_1$ or not according as the underlying Markov chain restricted to the communication class $C_i$ is positive recurrent or null recurrent, respectively. (In the latter case, the factor is actually of type $II_\infty$.) From this observation, Theorem~\ref{thm:ex_MC} follows with the help of Corollary~\ref{cor:Y_ctble}.}
\end{proof}

\begin{remark} \label{remark:LRD_sas_MC}
	\textnormal{In light of Remark~\ref{remark:LRD}, we perceive that the presence of more positive recurrent communication classes in the underlying Markov chain ensures longer memory of the associated  stationary stable process considered in Example~\ref{ex:sas_MC}. Intuitively, the positive recurrence of many states causes their frequent returns leading to the Markov chain's long range dependence, which gets transferred to the stable process via the random measure $M$.}
\end{remark}

\subsection{Conjectures and Open Problems} \label{sec:conj}

Our work leads to many conjectures and open problems. The most important one is the following.

\begin{op} \label{op:general_groups}
 Can we establish the $W^\ast$-rigidities in Theorem~\ref{thm:Wstar_rigidity} for a richer class $\mathcal{G}$ of countable groups? 
\end{op}

\noindent Observe that the class $\mathcal{G}$ in Theorem~\ref{thm:Wstar_rigidity} comes out of Theorems~\ref{thm:erg_charac} and \ref{thm:erg_part}, both of which can be extended to a bigger class of groups provided the same can be done for Theorem~\ref{thm:erg_charac_neveu}; see Remark~\ref{remark:positive_part_using_II_1}, which also states that the obstacle is ergodic theoretic and not operator algeraic. More precisely, the main hindrance of extending the proof of Theorem~3.1 of \cite{samorodnitsky:2005a} is the unavailability of an ergodic theorem for nonsingular actions of more general groups; see, for example, the dicussions in \cite{jarrett:2019}. 

The use of ergodic theorems in the proof of Theorem~3.1 of \cite{samorodnitsky:2005a} (and also Theorem~4.1 of \cite{wang:roy:stoev:2013}) is twofold. Firstly, it is used in establishing the criterion of ergodicity (and thereby linking it to weak mixing) as presented in Section~\ref{sec:erg_wm_nonerg} above. This needs a pointwise ergodic theorem for p.m.p.\ actions, which is actually available for any amenable group thanks to the seminal work of \cite{lindenstrauss:2001}. The role of hypercubes (for $\mathbb{Z}^d$) is played by a special family of F\o{}lner sets (called \emph{tempered F\o{}lner sequence}). The second use is slightly delicate. In order to mimic it, we need a stochastic ergodic theorem for a nonsingular action, which is not necessarily p.m.p.\ and this had to be established even for $G=\mathbb{Z}^d$ in Theorem~2.7 of \cite{wang:roy:stoev:2013} (see also \cite{hochman:2010} for a ratio ergodic theorem for multiparameter non-singular actions). 

In view of the discussions above, it is now clear that a stochastic ergodic theorem for nonsingular actions indexed by a more general group is the key to solving Problem~\ref{op:general_groups}. In partcular, for amenable groups, this has to be established along a tempered  F\o{}lner sequence (or its carefully chosen subsequence). For each discrete Heisenberg group, \cite{jarrett:2019} discovered a sequence of subsets (depending on the group) such that the ergodic theorem holds along the sequence for all nonsingular actions. This leads to the following conjecture.

\begin{conj}  \label{conj:heisenberg_group}
Theorem~\ref{thm:erg_charac_neveu} (and hence Theorems~\ref{thm:erg_charac}, \ref{thm:erg_part} and \ref{thm:Wstar_rigidity}) will hold for discrete Heisenberg groups.
\end{conj}

Resolution of Conjecture~\ref{conj:heisenberg_group} will, of course, solve Problem~\ref{op:general_groups} partially by enhancing the class $\mathcal{G}$ a bit. Apart from this, Conjecture~\ref{conj:heisenberg_group} will be of independent interest because discrete Heisenberg groups naturally crop up in space-time models. Also, random walks on such groups have been studied extensively; see, for example, \cite{beguin:valette:zuk:1997}, \cite{gretete:2011} and the references therein. Another natural and general question in this context is the following. 

\begin{op} \label{op:amenable}
Will Conjecture~\ref{conj:heisenberg_group} hold for groups of polynomial growth (or more generally, for a bigger class of amenable groups)?
\end{op}

Since many actions of finitely generated free groups $(\mathbb{F}_m,\; m \geq 2)$ and infinite property (T) groups (both are important classes of non-amenable groups) have $W^\ast$-rigidity (see, for example, \cite{popa:vaes:2014}) and $W^\ast$-superrigidity (see, for example, \cite{ioana:2011}), respectively, it is not unreasonable to expect that many properties (not just ergodicity or the complete absence of it) will become at least $W^\ast_m$-rigid for these class of groups. 

\begin{conj} \label{conj:free_prop_T}
 Many probabilistic properties of stationary \sas random fields will become $W^\ast_m$-rigid (a few properties will, in fact, be $W^\ast_R$-rigid) for $\mathcal{G}=\{\mathbb{F}_m: m \geq 2\}$ as well as for  $\mathcal{G}=\{G: G \mbox{ is a countably infinite group having property}~(T)\}.$ 
\end{conj}

\noindent Of course, Conjecture~\ref{conj:free_prop_T} is very broad and hence we can start with specific properties like ergodicity, weak mixing, mixing, etc. In particular, the nonsingular action of $\mathbb{F}_m$ on its Furstenberg-Poisson boundary $\partial\mathbb{F}_m$ is free and ergodic, and the corresponding group measure space construction is of type $III$. Therefore, the following conjecture is a natural one. 

\begin{conj} \label{conj:sas_bdy_action_free_gp}
The stationary \sas random field considered in Example~3.2 of \cite{sarkar:roy:2018} (generated by the boundary action of $\mathbb{F}_m$) is ergodic. 
\end{conj}

\noindent If Conjecture~\ref{conj:sas_bdy_action_free_gp} can be resolved positively, then the following question would be the next step.

\begin{op}
 Will a version of Theorem~\ref{thm:erg_charac} hold for any stationary \sas random field indexed by $\mathbb{F}_m$ at least under some conditions? 
\end{op}

The next problem is related to Question~\ref{qn_on_superrigidity_intro}. It is, of course, a very difficult question. Keeping this in mind, it is perhaps better to use various test cases and answer it partially.

\begin{op}\label{op:qn_superrigid}
Is it possible to answer Question~\ref{qn_on_superrigidity_intro} at least partially? In other words, is it possible to give conditions on the  group(s) and/or the action(s) and/or the stationary \sas field(s) such that the minimal (more generally, a Rosi\'nski) group measure space construction remembers the field fully or partially? 
\end{op}

\noindent Note that if all the conditions are put on exactly one of the two random fields, then the question of complete rememberance is a $W^\ast$-superrigidity question, which is beyond our reach at this point. Instead, we can take specific properties and ask the following question. 

\begin{op}
Will Conjecture~\ref{conj:free_prop_T} hold for the class $\mathcal{G}=\{\mathbb{Z}^d: d \in \mathbb{N}\}$ with respect to probabilistic properties other than ergodicity (or the absolute lack of it)? More specifically, will mixing be $W_R$-rigid or at least $W_m$-rigid for $\mathcal{G}$? 
\end{op}

Another special case of Problem~\ref{op:qn_superrigid} relates to the growth of partial maxima as investigated in \cite{samorodnitsky:2004a}, \cite{roy:samorodnitsky:2008}, \cite{sarkar:roy:2018} and \cite{athreya:mj:roy:2019}. Suppose $G$ is a finitely generated group and fixing a generating set of $G$, we study the growth maxima of a stationary $G$-indexed \sas random field on increasing balls of radius $n \to \infty$ with centre $e$ (the identity element of $G$) in the word metric.
\begin{op} \label{op:maxima}
Can we find a rich class $\mathcal{G}$ of finitely generated groups such that the growth of the maxima sequence (as mentioned above) being like the i.i.d.\ case will be a $W^\ast_R$-rigid (or at least $W^\ast_m$-rigid) property for $\mathcal{G}$? 
\end{op}

\noindent More specifically, we can restrict our attention to the class of groups considered in \cite{athreya:mj:roy:2019} and try to relate the non-vanishing cocycle growth of the underlying nonsingular action with the minimal (or a Rosi\'nski) group measure space construction. Alternatively, we can also try to solve Problem~\ref{op:maxima} for $\mathcal{G}= \{\mathbb{Z}^d: d \in \mathbb{N}\}$ and for $\mathcal{G}= \{\mathbb{F}_m: m \geq 2\}$ relating our work with that of \cite{roy:samorodnitsky:2008} and \cite{sarkar:roy:2018}, respectively. 

The next two problems, if positively answered, will show the strength of a Rosi\'nski group measure space construction by making it an invariant (in the first case) or at least an almost invariant (in the second case) for a stationary \sas random field.

\begin{op} \label{op:Rosinski_invariant}
	Can we establish, under additional conditions on the group and/or the actions and/or the random field, that any two Rosi\'nski group measure space constructions of a fixed stationary \sas random field are isomorphic as von Neumann algebras? 
\end{op}
 
\begin{op}
Even if we cannot solve Problem~\ref{op:Rosinski_invariant}, is it at least possible to reverse the implication in \eqref{impli_wstarm_wstarR} under some conditions?
\end{op}

\begin{remark}\label{remark:solns} \textnormal{As we see, this work leads to a Pandora's box of open problems and conjectures. Some of these have recently been resolved. For example, \cite{avraham-re’em:2023} proved that (1) - (3) of Theorem~\ref{thm:erg_charac} are equivalent for stationary \sas random fields indexed by any amenable group, not just $\mathbb{Z}^d$. Also, \cite{mj:roy:sarkar:2022} established the equvalence of (2) and (3) of Theorem~\ref{thm:erg_charac} for amenable group indexed \sas fields using a different approach and argued in Section~4.2 of their paper that their work, together with the work of \cite{avraham-re’em:2023}, solves Open~Problem~\ref{op:general_groups}, Conjecture~\ref{conj:heisenberg_group} and Open~Problem~\ref{op:amenable} above. In addition to this, \cite{mj:roy:sarkar:2022} also resolved Conjecture~\ref{conj:sas_bdy_action_free_gp} for stationary \sas random fields indexed by any non-elementary hyperbolic group, not just finitely generated free groups.}
\end{remark}

We hope to solve the rest of the problems in the years to come. Of course, trying out important test cases and then slowly making educated guesses will be the ideal approach towards this goal.

\section*{Acknowledgements} 
This paper is dedicated to the memory of Professor K.~R.~Parthasarathy, whose remarkable wisdom, endless encouragement and constant support have inspired generations of mathematicians in India. The author is extremely thankful to Daniel J. Rudolph for drawing his attention to Koopman representations (in 2005). He would also like to thank Sandeepan Parekh for explaining the basics of von Neumann algebras during his visit to the Indian Statistical Institute (in 2016), and Adrian Ioana for some important suggestions as well as for providing the reference \cite{vaes:2014}. The author is grateful to Apoorva Khare for the careful reading of a previous draft of this manuscript and his detailed comments that significantly improved the presentation of this paper. Seminars given by V.~S.~Sunder and Adrian Ioana helped the author understand the basics of the group measure space construction and $W^\ast$-rigidity theory, respectively. Discussions with Tathagata Banerjee, Riddhipratim Basu, B. V. Rajarama Bhat, Jyotishman Bhowmick, Nishant Chandgotia, Kajal Das, Debashish Goswami, Rajat Subhra Hazra, Manjunath Krishnapur, Mahan Mj, Jan Rosi\'{n}ski, Sourav Sarkar, Kalyan B.~Sinha and D. Yogeshwaran are also gratefully acknowledged.

\end{document}